\documentclass[11 pt, a4paper]{amsart}
\usepackage{a4,amsmath,amssymb,verbatim, enumerate}
\usepackage{tikz}

\usepackage{amssymb, amsmath, amsfonts, amsthm, graphics, xcolor}
\usepackage{ytableau} 
\usepackage{tikz}
\usepackage{marvosym}

\usepackage[colorlinks]{hyperref}

\usepackage{times}

\tikzstyle{V}=[draw, fill =black, circle, inner sep=0pt, minimum size=2pt]

\newtheorem{theorem}{Theorem}[section]
\newtheorem{proposition}[theorem]{Proposition}
\newtheorem{lemma}[theorem]{Lemma}
\newtheorem{corollary}[theorem]{Corollary}
\newtheorem{conjecture}[theorem]{Conjecture}

\theoremstyle{definition}
\newtheorem{example}[theorem]{Example}
\newtheorem{defi}[theorem]{Definition}
\newtheorem{remark}[theorem]{Remark}

\usepackage[enableskew]{youngtab}

\usepackage{color}

\newcommand{\blue}[1]{\color{blue}{#1}\color{black}}
\newcommand{\red}[1]{\color{red}{#1}\color{black}}

\newcommand{\ds}{\displaystyle}
\def\l{\lambda}
\def\m{\mu}
\def\n{\nu}
\def\a{\alpha}
\def\b{\beta}

\newcommand{\cir}{\circ}
\newcommand{\bu}{\bullet}

\newcommand{\rt}{\red{2}}
\newcommand{\rth}{\red{3}}
\newcommand{\rf}{\red{4}}
\newcommand{\rfi}{\red{5}}\newcommand{\ra}{\red{a}}\newcommand{\rb}{\red{b}}\newcommand{\rp}{\red{p}}
\newcommand{\rs}{\red{6}}
\newcommand{\rse}{\red{7}}
\newcommand{\bth}{\blue{3}}
\newcommand{\bfo}{\blue{4}}

\author[Ballantine]{Cristina Ballantine}
\thanks{This work was partially supported by a grant from the Simons Foundation (\#245997 to C. Ballantine)}
\address{Department of Mathematics and Computer Science, College of the Holy Cross, Worcester, MA 01610}
\email{cballant@holycross.edu}
\author[Orellana]{Rosa Orellana }
\thanks{R. Orellana was partially supported by the NSF Grant DMS-130512}
\address{Department of Mathematics,  Dartmouth College, 6188 Kemeny Hall, Hanover, NH 03755}
\email{Rosa.C.Orellana@Dartmouth.edu}

\title[Schur-Positivity in a Square]{Schur-Positivity in a Square}


\begin{document}

\begin{abstract}
Determining if a symmetric function is Schur-positive is a prevalent and, in general, a notoriously difficult problem.  In this paper we study the Schur-positivity of a family of symmetric functions.  Given a partition $\nu$, we denote by $\nu^c$ its complement in a square partition $(m^m)$.   We conjecture a  Schur-positivity criterion  for symmetric functions of the form $s_{\mu'}s_{\mu^c}-s_{\nu'}s_{\nu^c}$, where $\nu$ is a partition of weight $|\mu|-1$ contained in $\mu$ and the complement of $\mu$ is taken in the same square partition as the complement of $\nu$. We prove the conjecture in many cases. 

 \end{abstract}
\maketitle


\section*{Introduction}

The ring of symmetric functions has as a basis the Schur functions, $s_\lambda$, indexed by partitions $\lambda$. This basis is of particular importance in representation theory because its elements occur as characters of  the general linear group, $GL_n$, and they correspond to characters of the symmetric group, $S_n$,  via the Frobenius map.  In addition, the Schur functions are representatives of Schubert classes in the cohomology of the Grassmanian.  Often, given a symmetric function, we are  interested in writing it in the Schur basis.   If the coefficients in the Schur basis expansion are all non-negative integers, then the symmetric function corresponds to the character of a representation of $GL_n$ or $S_n$. In this case, the coefficients are simply giving the decomposition of the character in terms of irreducible characters.  We will call a symmetric function {\em Schur-positive} if it is a linear combination of Schur functions with non-negative coefficients.

In recent years there has been increased interest in studying the Schur-positivity of expressions of the form 
\begin{equation}\label{generaleq}
s_\lambda s_\mu - s_\alpha s_\beta.
\end{equation}
See for example \cite{BBR, BM, FFLP, KWW, Ki,  LPP, O}.   These expressions  can also be interpreted as differences of skew Schur functions which have been studied in  \cite{MvW1, MvW2}.  The Schur-positivity of such expressions is equivalent to inequalities between Littlewood-Richardson coefficients.  In this paper, we study the Schur-positivity of  a family of  expressions of this form.  Let $\nu$ and $\mu$ be  partitions such that $|\mu|=|\nu|+1$ and $\nu \subseteq \mu$, and  let $m$ be an integer such that $m \geq \mu_1+\ell(\mu)$.   We denote by  $\nu^c$ (respectively $\mu^c$)  the complement of $\nu$ (respectively $\mu$) in the square partition $(m^m)$ and by $\nu'$ (respectively $\mu'$) the conjugate of $\nu$ (respectively $\mu$).  We are interested in the Schur-positivity of  expressions of the form
\begin{equation} \label{oursymfun}
s_{\mu'}s_{\mu^c} - s_{\nu'}s_{\nu^c}.
\end{equation}
As we explain in section \ref{prod}, these expressions arise in the study of the Kronecker product  of a square shape and a hook shape.   Determining the Schur decomposition of this Kronecker product is of particular interest in a paper by Scharf, Thibon and Wybourne \cite{STW}  on the powers of the Vandermonde determinant and its application to the quantum Hall effect.  In their paper, the $q$-discriminant is expanded as a linear combination of Schur functions where the coefficients are specializations of the Kronecker products mentioned above. 

Central to our paper is a  conjectural criterion  for the Schur-positivity of  (\ref{oursymfun}). \vspace{.05in}

\noindent {\bf Conjecture:} Let $\nu\vdash n$ and $\mu \vdash n+1$  be a partitions  such that $\nu\subseteq \mu$. Then, if complements are taken in a large enough square, $s_{\mu'}s_{\mu^c}-s_{\nu'}s_{\nu^c}$ is Schur-positive if and only if the following conditions are satisfied. 
\begin{enumerate}
\item[(C1)] $\nu$ or $\nu'$ is of the form $\beta+(s^s)+\alpha$, where $\b$ and $\a$ are partitions such that,  \begin{enumerate}

\item[(i)] if  $\beta_1=i$ , then $\b$  contains $(i^{s+1}, i, i-1, i-2, \ldots, 1)$,

\item[(ii)] $s\neq0$ if and only if $\a \neq \emptyset$.

\end{enumerate}

\item[(C2)]  If $\nu$ (respectively $\nu'$) is a partition $\beta+(s^s)+\alpha$ as in (C1), then $\mu$ (respectively $\mu'$) is the partition $\beta+(s^s,1)+\alpha$.
\end{enumerate}


If $\nu$ satisfies (C1), we say that $\nu$ is of \textbf{\textit{type $\mathbf{1}$}}  if $s=0$ and of \textbf{\textit{type $\mathbf{2}$}} if $s\geq 1$. 

In \cite{PM} McNamara gives several necessary conditions for Schur-positivity. However, most articles  considering special cases of (\ref{generaleq}) focus on sufficient conditions.    The strength of our conjecture lies in the fact that is it a criterion. In this article, we prove the conjecture in many cases. 


All results of this paper are  valid if all complements are taken in a large enough rectangle instead of a square. Since the proofs for complements in rectangles do not add any new insight, to simplify the exposition, we present the results with complements taken in a large enough square. 

The paper is organized as follows. In section 1 we review the notation and basic facts about partitions and Schur functions and discuss products of the form $s_{\mu'}s_{\mu^c}$. In section 2 we discuss symmetry and stability properties of expressions of the form (\ref{oursymfun}), we introduce the main conjecture, and discuss type possibilities for a partition and its conjugate. In section 3 we study partitions of type 1 and show that the conjecture holds if the partition contains a large rectangle of the same width as the partition. Moreover, we show that the conjecture holds for all partitions of type 1 of width at most  four. In section 3 we also conjecture the smallest partition in lexicographic order appearing in the decomposition of $s_{\nu'}s_{\nu^c}$ with $\nu$ of type 1. The validity of this conjecture  implies that  condition (C2)  on  $\mu$ is necessary for the Schur-positivity of $s_{\mu'}s_{\mu^c}-s_{\nu'}s_{\nu^c}$  if $\nu$ satisfies (C1). In section 4 we establish several properties of  partitions of type 2 and prove the conjecture for partitions of type 2 with $\b_1=0$ or $1$. In section 5 we consider the failure of Schur-positivity for the symmetric function  (\ref{oursymfun}) if $\nu$ does not satisfy (C1) (\textit{i.e.}, $\nu$ is neither of type 1, nor of type 2). In section 6 we offer some final remarks, including unsuccessful strategies we considered in our attempt to prove the conjecture.

\subsection*{Acknowledgements:} The first author would like to thank Vyjayanthi Chari, Thomas Lam, Ronald King, and Arun Ram for helpful discussions. We are particularly grateful to Nicolas Thi\'ery for his help calculating the Schur decomposition in large examples. These calculations helped us  formulate the lex-minimality  conjecture in section 3.3.

\section{Preliminaries}

\subsection{Partitions and Schur functions}

For details and proofs of the contents presented here see \cite{ma} or
\cite[Chapter 7]{stanleyEC2}.  Let $n$ be a non-negative integer. A \emph{partition} of  $n$
is a weakly decreasing sequence of non-negative integers,
$\l:=(\l_1,\l_2,\cdots,\l_{\ell})$, such that $|\l|=\sum \l_i=n$. We write $\l\vdash
n$ to mean $\l$ is a partition of $n$. The nonzero integers $\lambda_i$ are called
the \emph{parts} of $\l$. We identify a partition with its \emph{Young diagram},
i.e. the array of left-justified squares (boxes) with $\l_1$ boxes in the first row,
$\l_2$ boxes in the second row, and so on. The rows are arranged in matrix form from
top to bottom. By the box in position $(i,j)$ we mean the box  in the $i$-th row and
$j$-th column of $\l$. The \emph{length} of $\l$, $\ell(\l)$, is the number of rows
in the Young diagram. 

\begin{figure}[ht]
\centering
{\tiny \yng(6,4,3,1)}
\caption{ $\l=(6,4,3,1), \ \ \ \ell(\l)=4, \ \ \ |\l|=14$}
\end{figure}
Given two partitions $\l$ and $\m$, we write $\m\subseteq \l$ if and only if
$\ell(\m) \leq \ell(\l)$ and $\m_i\leq \l_i$ for all $1\leq i\leq \ell(\m)$. If $\m
\subseteq \l$, we denote by $\l/\m$ the skew shape obtained by removing the boxes
corresponding to $\m$ from $\l$.

\begin{figure}[ht]
\centering
{\tiny \young(:::\hfil\hfil\hfil,:\hfil\hfil\hfil,:\hfil\hfill,\hfil)}
\caption{  $\l/\m$ with $ \l=(6,4,3,1)$  and  $\m=(3,1,1)$ }
\end{figure}

 For any two Young diagrams (or skew shapes) $\l$ and $\m$, we denote by  $\l \times \m$  any skew diagram consisting only of the diagram $\l$ followed by the diagram $\mu$ such that $\l$ and $\mu$ have no common edges.  That is, the rows (respectively  columns) of $\l$ are above the rows (respectively to the right of  the columns) of $\m$.  Note that $\l$ and $\m$ could have one common corner, \textit{i.e},  the furthest northeast corner of $\m$ can connect with the lowest southwest corner of $\l$.

\begin{figure}[ht]
\centering
{\tiny \young(::::\hfil\hfil,::::\hfil,\hfil\hfil\hfil,\hfil\hfil)} 
\caption{ $\l\times \m$ with $\l=(2,1)$ and $\m=(3,2)$.}
\end{figure}
 
A \emph{semi-standard Young tableau} (SSYT) \emph{of shape} $\l/\m$ is a filling of
the boxes in the Young diagram of the skew shape $\l/\m$ with positive integers so that the numbers
weakly increase in each row from left to right and strictly increase in each column
from top to bottom. The \emph{type} of a SSYT $T$ is the sequence of non-negative
integers $(t_1,t_2,\ldots)$, where $t_i$ is the number of labels $i$ in $T$. 
The \textit{superstandard tableau} of a partition $\l$ is the SSYT of shape $\l$ and type $\l$. 
\begin{figure}[ht]
\centering
{\tiny \young(:::2234,::1446,:1366,224)} 
\caption{  SSYT of shape $\l/\m=(7,6,5,3)/(3,2,1)$ and type $(2,4,2,4,0,3)$.}
\end{figure}

\vskip 0in Given a SSYT $T$ of shape $\l/\m$ and type $(t_1,t_2,\ldots)$, we define
its \emph{weight}, $w(T)$, to be the monomial obtained by replacing each $i$ in $T$
by $x_i$ and taking the product over all boxes, i.e.
$w(T)=x_1^{t_1}x_2^{t_2}\cdots$. For example, the weight of the SSYT in Figure 4 is
$x_1^2x_2^4x_3^2x_4^4x_6^3$. The skew Schur function $s_{\l/\m}$ is defined
combinatorially by the formal power series
$$s_{\l/\m}= \sum_T w(T),$$
where the sum runs over all SSYTs of shape $\l/\m$. To obtain the usual Schur
function one sets $\m =\emptyset$.  It
follows directly from the combinatorial definition of Schur functions  that $s_{\l
\times \m} = s_{\l}s_{\m}$. \vskip 0in The space of homogeneous symmetric
functions of degree $n$ is denoted by $\Lambda^n$. A basis for this space is given
by the Schur functions $\{ s_\l\,|\, \lambda\vdash n\}$. The Hall inner product on
$\Lambda^n$ is denoted by $\langle \ , \ \rangle$ and it is defined by
$$\langle s_{\l},s_\m\rangle=\delta_{\l\m},$$
 where  $\delta_{\l\m}$  denotes the Kronecker delta. 
 
 
 \subsection{The Littlewood-Richardson Rule}
 The \emph{Littlewood-Richardson
coefficients} are defined via the Hall inner product on symmetric functions as
follows:
$$c_{\m\, \n}^{\l} := \langle s_\l, s_\m s_\n \rangle=\langle s_{\l/\m},
s_\n \rangle. $$ That is, skewing is the adjoint operator of multiplication with
respect to this inner product. The Littlewood-Richardson coefficients can be 
computed using the Littlewood-Richardson rule.  Before presenting the
rule we need to recall two additional notions. A \emph{lattice permutation} is a
sequence $a_1a_2\cdots a_n$ such that in any initial factor $a_1a_2\cdots a_j$, for each $1\leq i \leq n$, the
number of labels $i$ is at least  the number of labels $(i+1)$. For
example $11122321$ is a lattice permutation. The \emph{reverse reading word} of a
tableau is the sequence of entries of $T$ obtained by reading the entries from right
to left and top to bottom, starting with the first row.
 \vskip 0.05in

 \noindent \textbf{Example:} The reverse reading word of the tableau {\tiny
\Yvcentermath1$\young(::12,3568,479)$} is $218653974$. \vskip 0.05in 

We denote by $LR(\l/\mu,\nu)$ the collection of SSYTs of shape $\l/\m$ and
type $\n$ whose reverse reading word is a lattice permutation. We refer to a tableau in $LR(\l/\mu,\nu)$ as a \textit{Littlewood-Richardson tableau} of shape $\l/\m$ and
type $\n$. The \emph{Littlewood-Richardson rule} states that the Littlewood-Richardson
coefficient $c_{\m\, \n}^{\l}$ is equal to the cardinality of $LR(\l/\m,\nu)$.  \vskip 0in 

We denote by $c_{\m\,  \n\,  \eta}^{\l}$ the multiplicity of $s_{\l}$ in the product $s_{\m}s_{\n}s_{\eta}$, \textit{i.e.}, $c_{\m\,  \n\,  \eta}^{\l}=\langle s_{\l}, s_{\m}s_{\n}s_{\eta}\rangle$.

In the next remark, we draw attention to a condition on tableaux in $LR(\l/\m,\nu)$ which will be used later in our discussion. First we introduce a definition.

Recall that given a partition $\nu$, a \textit{horizontal strip} in $\nu$ is a skew shape $\nu/\psi$ (for some partition $\psi\subseteq \nu$) such that no two boxes in $\nu/\psi$ are in the same column. 

\begin{remark}\label{horizontal-strips} The lattice permutation condition on a tableau  $T\in LR(\l/\m,\nu)$ imposes the following condition. Let $\nu^{(1)}=\nu$. The labels in the last row of $T$ must form a horizontal strip in the superstandard tableau of $\nu^{(1)}$. Denote this horizontal strip by $h_1$. Note that if we list the labels in the last row of $T$ in order from left to right, they appear in $h_1$ in that order when read from right to left. Let $\nu^{(2)}$ be the partition obtained from $\nu^{(1)}$ by removing the boxes of $h_1$. Then, the labels in the second to last row of $T$ form a horizontal strip in the superstandard tableau of $\nu^{(2)}$. This process continues recursively. We denote by $h_j$ the horizontal strip in $\nu^{(j)}$ containing the labels of the $j$-th row form the bottom in $T$ and let $\nu^{(j+1)}$ be the partition obtained form $\nu^{(j)}$ by removing the boxes of $h_j$. Then, the labels in the $(j+1)$-st row form the bottom in $T$ must form a horizontal strip in the superstandard tableau of $\nu^{(j+1)}$. 

\end{remark}

We refer to the condition in Remark \ref{horizontal-strips} as \textit{the rows of $T$  form horizontal strips in the superstandard tableau of $\nu$}.

{\small
\begin{figure}[ht]
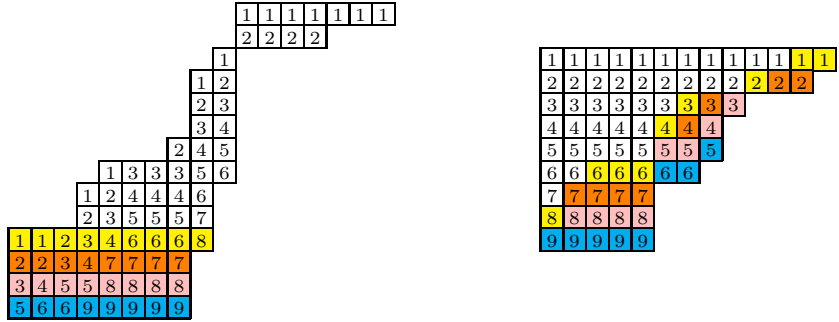

\centering
\ytableausetup{smalltableaux}
\begin{ytableau}
\none &\none &\none &\none& \none &\none &\none &\none& \none& \none &  *(white) 1&  *(white) 1&  *(white) 1&  *(white) 1&  *(white) 1&  *(white) 1&  *(white) 1\\
\none &\none &\none &\none& \none &\none &\none &\none& \none& \none & *(white) 2& *(white) 2& *(white) 2& *(white) 2\\
\none &\none &\none &\none& \none &\none &\none &\none& \none &  *(white) 1\\
\none &\none &\none& \none &\none &\none &\none& \none &  *(white) 1 & *(white) 2\\
\none &\none &\none& \none &\none &\none &\none& \none &  *(white) 2 & *(white) 3\\
\none &\none &\none& \none &\none &\none & \none &\none & *(white) 3 & *(white) 4\\
\none &\none &\none& \none &\none & \none &\none  &  *(white) 2 & *(white) 4 & *(white) 5\\
\none &\none& \none &\none & *(white) 1 & *(white) 3 & *(white) 3 &  *(white) 3 & *(white) 5 & *(white) 6\\
\none& \none &\none &*(white) 1 & *(white) 2 & *(white) 4 & *(white) 4 &  *(white) 4 & *(white) 6 \\
 \none &\none &\none &*(white) 2 & *(white) 3 & *(white) 5 & *(white) 5 &  *(white) 5 & *(white) 7 \\
*(yellow) 1 &*(yellow) 1 &*(yellow) 2 &*(yellow) 3 & *(yellow) 4 & *(yellow) 6 & *(yellow) 6 &  *(yellow) 6 & *(yellow) 8\\
*(orange) 2 &*(orange) 2 &*(orange) 3 &*(orange) 4 & *(orange) 7 & *(orange) 7 & *(orange) 7 &  *(orange) 7 \\
*(pink) 3 &*(pink) 4 &*(pink) 5 &*(pink) 5 & *(pink) 8 & *(pink) 8 & *(pink) 8 &  *(pink) 8  \\
*(cyan) 5 &*(cyan) 6 &*(cyan) 6 &*(cyan) 9 & *(cyan) 9 & *(cyan) 9 & *(cyan) 9 &  *(cyan) 9  \\
\end{ytableau}
\hspace{50pt}
\begin{ytableau}
\none \\
\none \\
 *(white) 1 & *(white) 1 & *(white) 1 & *(white) 1 & *(white) 1 & *(white) 1 & *(white) 1 & *(white) 1 & *(white) 1 & *(white) 1 & *(white) 1 &*(yellow) 1 & *(yellow) 1\\
*(white) 2 & *(white) 2 & *(white) 2 & *(white) 2 & *(white) 2 & *(white) 2 & *(white) 2 & *(white) 2 & *(white) 2 &*(yellow) 2 & *(orange) 2 & *(orange) 2 \\
*(white) 3 & *(white) 3 & *(white) 3 & *(white) 3 & *(white) 3 & *(white) 3 & *(yellow) 3 & *(orange) 3 & *(pink) 3  \\
*(white) 4 & *(white) 4 & *(white) 4 & *(white) 4 & *(white) 4 & *(yellow) 4 & *(orange) 4 & *(pink) 4 \\
*(white) 5 & *(white) 5 &*(white) 5 &*(white) 5 &*(white) 5 &*(pink) 5 &*(pink) 5 &*(cyan) 5 \\
*(white) 6 & *(white) 6 & *(yellow) 6 & *(yellow) 6 & *(yellow) 6 & *(cyan) 6 & *(cyan) 6 \\
*(white) 7 & *(orange) 7 & *(orange) 7 & *(orange) 7 & *(orange) 7 \\ 
*(yellow) 8 & *(pink) 8 & *(pink) 8 & *(pink) 8 & *(pink) 8 \\
*(cyan) 9 & *(cyan) 9 & *(cyan) 9 & *(cyan) 9 & *(cyan) 9 \\
\end{ytableau}
\caption{ A Tableau $T$  and the corresponding superstandard tableau}\label{fig:h-s}
\end{figure}}

Figure \ref{fig:h-s} shows, on the left, a  tableau $T\in LR(\l/\mu, \nu)$ with $\l=(17, 14, 10^6, 9^3, 8^3)$, $\mu=(10^2, 9, 8^3, 7, 4, 3^2)$, and $\nu=(13, 12, 9, 8^2, 7, 5^3)$. On the right, it shows the superstandard tableau of $\nu$. 
The colors indicate  the labels in some of the bottom rows of $T$ and the corresponding horizontal strips in the superstandard tableau of $\nu$. 

\subsection{The product of the conjugate and the complement of a partition}\label{prod}

Let $\nu=(\nu_1, \nu_2, \ldots, \nu_{\ell})$ be a partition. We denote by $\nu'$ the \emph{conjugate partition} of $\nu$, i.e., the partition whose rows are the columns of $\nu$. If $D$ is a skew-diagram,  $D^*$ denotes $D$ rotated by $180^\circ$. 


If $(m^m)$ is a square partition and $\mu \subseteq (m^m)$, the \emph{complement partition} of $\mu$ in $(m^m)$, denoted $\mu^c$, is the partition   $((m^m)/\mu)^*$. Whenever we need to emphasize $m$, we write $\mu^{c,m}$ for $\mu^c$.  See Figure \ref{fig:complement} for the shape of the complement. 

\begin{figure}[ht]
\centering
\begin{picture}(130,110)
\put(0,0){\line(0,0){100}}
\put(0,0){\line(1,0){70}}
\put(0,100){\line(1,0){110}}
\put(80,15){\line(0,0){15}}
\put(80,30){\line(1,0){10}}
\put(90,30){\line(0,0){10}}
\put(90,40){\line(1,0){20}}
\put(110,40){\line(0,0){60}}
\put(70,0){\line(0,0){15}}
\put(70,15){\line(1,0){10}}
\put(95,15){$\mu^*$}
\put(-13,45){$m$}
\put(-7,0){\line(0,0){35}}
\put(-10,0){\line(1,0){6}}
\put(-7,60){\line(0,0){40}}
\put(-10,100){\line(1,0){6}}
\put(50,105){$m$}
\put(0,107){\line(1,0){45}}
\put(0,104){\line(0,0){6}}
\put(65,107){\line(1,0){45}}
\put(110,104){\line(0,0){6}}
\end{picture}
\caption{The complement of $\mu$, i.e., $\mu^c$.}
\label{fig:complement}
\end{figure}
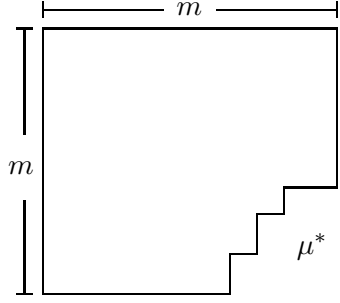

Given a partition $\m$ and a square $(m^m)$, we will be interested in the product $s_{\m'}s_{\m^c}$.

Our interest in this particular product has its origin in  the Kronecker product of a hook and a square shape, \text{i.e.}, $s_{(n-k,1^k)}\ast s_{(m^m)}$.  We briefly explain this connection.

Using Pieri's rule and induction, one can easily see that $$\ds s_{(n-k,1^k)} =  \sum_{i=0}^k (-1)^i s_{(n-k+i)} s_{(1^{k-i})}.$$ Then, using Littlewood's formula \cite{L}, we obtain
\[(s_{(n-k+i)} s_{(1^{k-i})})\ast s_{(m^m)} = \sum_{\mu\vdash n-k+i, \nu\vdash k-i} c_{\mu\, \nu}^{(m^m)} s_\mu s_{\nu'}.\]
Since $\ds c_{\mu,\nu}^{(m^m)} = \left\{ \begin{array}{ll}1 & \mbox{ \ if \ } \mu, \nu\subset (m^m) \mbox{\  and \ }\nu =\mu^c\\ 0 & \mbox{\ else.} \end{array} \right.$, we have 
\[s_{(n-k,1^k)}\ast s_{(m^t)}=\sum_{\eta \vdash n} \left( \sum_{i=0}^k (-1)^{k-i}\sum_{\mu \vdash i}c_{\mu'\mu^c}^{\eta} \right)s_{\eta}.\]

The  focus of this article article is the study of Schur-positivity of differences of  products of the form $s_{\m'}s_{\m^c}$. 

Notice that if we choose $m$ large, then the partitions  occurring in the expansion of $s_{\m'}s_{\m^c}$ have the shape shown in Figure \ref{fig:shape-of-partitions}.  That is, if $c_{\m'\m^c}^\lambda\neq 0$, then $\lambda = (\eta^c + \gamma, \sigma)$, where $\eta \subseteq \mu$, $|\gamma|+|\sigma|=|\eta|$ and the complement of $\eta$ is taken in $(m^m)$.

\begin{figure}[ht]
\centering
\begin{picture}(130,110)
\put(0,-10){\line(0,0){110}}
\put(0,20){\line(1,0){80}}
\put(0,100){\line(1,0){140}}
\put(80,20){\line(0,0){10}}
\put(80,30){\line(1,0){10}}
\put(90,30){\line(0,0){10}}
\put(90,40){\line(1,0){20}}
\put(110,40){\line(0,0){60}}
\put(95,25){$\eta^*$}
\put(140,100){\line(0,-1){10}}
\put(140,90){\line(-1,0){5}}
\put(135,90){\line(0,-1){10}}
\put(135,80){\line(-1,0){10}}
\put(125,80){\line(0,-1){10}}
\put(125,70){\line(-1,0){15}}
\put(120,85){$\gamma$}
\put(0,-10){\line(1,0){10}}
\put(10,-10){\line(0,1){5}}
\put(10,-5){\line(1,0){10}}
\put(20,-5){\line(0,1){10}}
\put(20,5){\line(1,0){10}}
\put(30,5){\line(0,1){15}}
\put(10,5){$\sigma$}
\put(-13,60){$m$}
\put(-7,20){\line(0,0){35}}
\put(-10,20){\line(1,0){6}}
\put(-7,70){\line(0,0){30}}
\put(-10,100){\line(1,0){6}}
\put(50,105){$m$}
\put(0,107){\line(1,0){45}}
\put(0,104){\line(0,0){6}}
\put(65,107){\line(1,0){45}}
\put(110,104){\line(0,0){6}}
\end{picture}
\caption{ The shape of the partition $\lambda=(\eta^c+\gamma, \sigma)$ for large $m$.}
\label{fig:shape-of-partitions}
\end{figure}
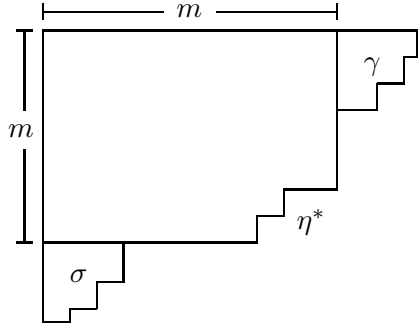

\begin{remark}\label{stability:product}
In order for $\lambda$ to occur with  nonzero coefficient  in the expansion of $s_{\m'}s_{\m^c}$, there must be LR tableaux of shape $\lambda/\m^c$ and type $\m'$.  With the notation of Figure \ref{fig:shape-of-partitions}, we must have  $\ell(\gamma)\leq \mu_1$ and $\sigma_1\leq \ell(\m)$.  In addition,  $\eta_1\leq \m_1$ and $\ell(\eta) \leq \ell(\m)$.   Therefore, if $m\geq \mu_1+\ell(\mu)$,  the skew diagram $\lambda/{\m^c}$ is the disjoint union of three distinct diagrams, $\gamma, (\mu/\eta)^*$, and $\sigma$, and the label fillings in Littlewood-Richardson tableaux do not depend on $m$.  Hence, the decomposition of the product $s_{\m'}s_{\m^c}$ becomes stable, i.e. does not change  for  $m\geq \mu_1+\ell(\mu)$.
\end{remark}

For the rest of the article  we  assume that complements are taken in a square $(m^m)$ with  $m\geq \mu_1+\ell(\mu)$.   Thus, we  assume that all partitions $\lambda$ such that $\ds c_{\mu'\mu^c}^{\l}\neq 0$, have the shape in Figure \ref{fig:shape-of-partitions}, \textit{i.e.}, $\lambda = (\eta^c + \gamma, \sigma)$, where $\eta \subseteq \mu$ and $|\gamma|+|\sigma|=|\eta|$.

The following proposition describes a symmetry property for the expansion of the product $s_{\m'}s_{\m^c}$. 

\begin{proposition} \label{sym} Let $\mu$ be a partition and $m\in \mathbb{Z}$ with $m \geq \mu_1+\ell(\mu)$. 
Consider the partitions $\lambda=(\eta^c+\gamma,\sigma)$ and $\overline{\lambda}=(\eta^c+\sigma,\gamma)$ with $\eta, \gamma, \sigma$ as above. The coefficients of $s_\lambda$ and $s_{\overline{\lambda}}$ in $s_{\mu'}s_{\mu^c}$  are equal. Moreover, these coefficients are equal to $c^{\mu'}_{\gamma\, \sigma\,(\mu/\eta)^*}$.

\end{proposition} 
\begin{proof} If $\l$, $\mu$ and $m$  are as in the statement of the proposition, we have $$\ds \langle s_{\lambda},s_{\mu'}s_{\mu^c}\rangle= \langle s_{\lambda/\mu^c},s_{\mu'} \rangle=\langle s_{\gamma}s_{(\mu/\eta)^*}s_{\sigma},s_{\mu'}\rangle=\langle s_{\gamma}s_{\sigma}s_{(\mu/\eta)^*},s_{\mu'}\rangle.$$ Thus, $\ds \langle s_{\lambda},s_{\mu'}s_{\mu^c}\rangle=c^{\mu'}_{\gamma\, \sigma\, (\mu/\eta)^*}$. Moreover,  $$\label{switch} \langle s_{\lambda},s_{\mu'}s_{\mu^c}\rangle=\langle s_{\sigma}s_{(\mu/\eta)^*}s_{\gamma},s_{\mu'}\rangle= \langle s_{\overline{\lambda}/\mu^c},s_{\mu'} \rangle= \langle s_{\overline{\lambda}},s_{\mu'}s_{\mu^c}\rangle.$$ 
 \end{proof}  

A consequence of Proposition \ref{sym} is the following combinatorial interpretation of the coefficient of $s_{\l}$ in $s_{\mu'}s_{\mu^c}$. 

\begin{corollary}\label{cor-sym} If $\l$, $\mu$, and $m$ are as in Proposition \ref{sym}, then $$\langle s_{\l},s_{\mu'}s_{\mu^c}\rangle=|LR(\gamma\times\sigma\times (\mu/\eta)^*, \mu')|.$$
\end{corollary}


\section{Schur-positivity  for a difference of products}

\bigskip

We say that a symmetric function $f$ is {\em Schur-positive} if every  coefficient in the expansion of $f$ as a linear combination of Schur functions is a  non-negative integer.  That is, if 
\[f= \sum_{\lambda} c_\lambda s_\lambda,\]
then $c_\lambda\geq 0$ for all $\lambda$.  

In this paper we are interested in the Schur-positivity of expressions of the form 
\[s_{\mu'} s_{\mu^c} - s_{\nu'}s_{\nu^c},\] 
where $\nu^c$ and $\mu^c$ are complements in the same square, $\nu \subseteq \mu$,  and $\mu \vdash |\nu|+1$.  We  introduce another definition to simplify the statements of the theorems.

\begin{defi}We say that $\mu$ \emph{covers} $\nu$,  if $\nu \subseteq \mu$, $|\mu|=|\nu|+1$, $\nu^c$ and $\mu^c$ are complements in the same square,  and $s_{\mu'}s_{\mu^c}-s_{\nu'}s_{\nu^c}$ is Schur-positive. \end{defi} 

For the rest of the article, whenever complements of different partitions appear in the same expression, it is understood that they are taken in the same square. 

In the proof of the following proposition we use the \emph{involution} $\omega$ on the ring of symmetric functions, i.e., $\omega: \Lambda\rightarrow \Lambda$. This is defined on the Schur functions by $\omega(s_\lambda) = s_{\lambda'}$.  It is a well-known fact that $\omega$ is an isometry with respect to the Hall inner product.  For  details  see \cite[Section 7.6]{stanleyEC2}. In particular, $\omega$ is an algebra endomorphism. This leads immediately to the following result.

\begin{proposition}\label{conj}
If $\mu$ covers $\nu $, then $\mu'$ covers $\nu'$.

\end{proposition}

\subsection{Symmetry and Stability} In this section we describe two properties of the coefficients  occurring in the expansion of 
$s_{\mu'}s_{\mu^c}- s_{\nu'}s_{\nu^c}$.   

\begin{proposition}[Symmetry] \label{cor-sym2} Let $\mu$ be a partition and $m\in \mathbb{Z}$ with $m \geq\mu_1+\ell(\mu)$.  If $\lambda=(\eta^c+\gamma,\sigma)$ and $\overline{\lambda}=(\eta^c+\sigma,\gamma)$,   then the coefficients of $s_\lambda$ and $s_{\overline{\lambda}}$ in  $s_\mu's_{\mu^c}-s_\nu's_{\nu^c}$ are equal. 
\end{proposition}
\begin{proof}
 If $m \geq\mu_1+\ell(\mu)$, by Proposition \ref{sym},  both products $s_{\m'}s_{\m^c}$ and $s_{\n'}s_{\n^c}$ are symmetric.  
\end{proof}

For  fixed $\nu$ and $\mu$, the differences $s_{\mu'}s_{\mu^c}-s_{\nu'}s_{\nu^c}$  satisfy the stability property given in the following proposition. 

\begin{proposition}[Stability] \label{stab} If $\nu$ and $\mu$ are partitions such that $\nu \subseteq \mu$ and $|\mu|=|\nu|+1$, then $s_{\mu'}s_{\mu^c}-s_{\nu'}s_{\nu^c}$ is stable in the sense that if   $m_1, m_2 \geq \mu_1+\ell(\mu)$, then $$\langle s_{(\eta^{c,{m_1}}+\gamma, \sigma)},  s_{\mu'}s_{\mu^{c,{m_1}}}-s_{\nu'}s_{\nu^{c,{m_1}}}\rangle=\langle s_{(\eta^{c,{m_2}}+\gamma, \sigma)},  s_{\mu'}s_{\mu^{c,{m_2}}}-s_{\nu'}s_{\nu^{c,{m_2}}}\rangle,$$ for all $\eta, \gamma, \sigma$ with $|\gamma|+|\sigma|=|\eta|$. 
\end{proposition}

\begin{proof} This follows directly from Remark \ref{stability:product} and Corollary \ref{cor-sym} since $\langle s_{\l},s_{\mu'}s_{\mu^c}\rangle=|LR(\gamma\times\sigma\times (\mu/\eta)^*, \mu')|$ does not depend on $m$ if $m \geq \mu_1+\ell(\mu)$.
\end{proof}

The following example shows that the bound $\mu_1+\ell(\mu)$ on $m$ in Proposition \ref{stab} is sharp. 

\begin{example}
Suppose $\mu=(3,2,1)$ and $\nu=(3,2)$. Thus, $\mu_1=3$ and $\ell(\mu)=3$. Let $\eta=(2,2,1), \gamma=(1,1,1)$, and $\sigma=(1,1)$. One can use Maple to check that, if $m_1=6$, then $\ds \langle s_{(\eta^{c,{m_1}}+\gamma, \sigma)},  s_{\mu'}s_{\mu^{c,{m_1}}}-s_{\nu'}s_{\nu^{c,{m_1}}}\rangle=1$. However, if $m_2=5$,  $(\eta^{c,{m_2}}+\gamma, \sigma)$ is not a partition and thus  there is no Schur function $ s_{(\eta^{c,{m_2}}+\gamma, \sigma)}$ in the expansion of $s_{\m'}s_{\m^{c,m_2}} - s_{\n'}s_{\n^{c,m_2}}$. 

\end{example}


\subsection{Main Conjecture} 
We now introduce our main conjecture. It gives a characterization of partitions $\nu$ and $\mu$ for which  $s_{\m'}s_{\m^c} - s_{\n'}s_{\n^c}$ is Schur-positive (assuming complements are taken in sufficiently large squares). 


\begin{conjecture}\label{conjecture}
Let $\nu\vdash n$ and $\mu \vdash n+1$  be a partitions  such that $\nu\subseteq \mu$. Suppose complements are taken in $(m^m)$ with $m\geq \mu_1+\ell(\mu)$.  Then, $\mu$ covers $\nu$ if and only if the following conditions are satisfied. 
\begin{enumerate}
\item[\text{\emph{(C1)}}] $\nu$ or $\nu'$ is of the form $\beta+(s^s)+\alpha$, where $\b$ and $\a$ are partitions such that,  \begin{enumerate}

\item[(i)] if  $\beta_1=i$ , then $\b$  contains $(i^{s+1}, i, i-1, i-2, \ldots, 1)$,

\item[(ii)] $s\neq0$ if and only if $\a \neq \emptyset$.

\end{enumerate}

\item[\text{\emph{(C2)}}]  If $\nu$ (respectively $\nu'$) is a partition $\beta+(s^s)+\alpha$ as in \text{\emph{(C1)}}, then $\mu$ (respectively $\mu'$) is the partition $\beta+(s^s,1)+\alpha$.
\end{enumerate}
\end{conjecture}

Suppose $\nu$ satisfies (C1).  We say that $\nu$ is of \textbf{\textit{type $\mathbf{1}$}}  if $s=0$ and of \textbf{\textit{type $\mathbf{2}$}} if $s\geq 1$. 

 Note that in (C1) it is possible to have $\b=\emptyset$. 

In Figure \ref{fig:type1}, we show  a general partition $\nu$ of type 1 together with the corresponding cover $\mu$. Inside the diagram of $\nu$ we show (in blue) the smallest partition of type 1 of  width equal to the width of $\nu$.  In Figure \ref{fig:type2} we show a general partition of type 2 together with the corresponding cover $\mu$.  In both cases the red box represents the  box added to $\nu$ to obtain $\mu$.  We will see below that this box is unique. 

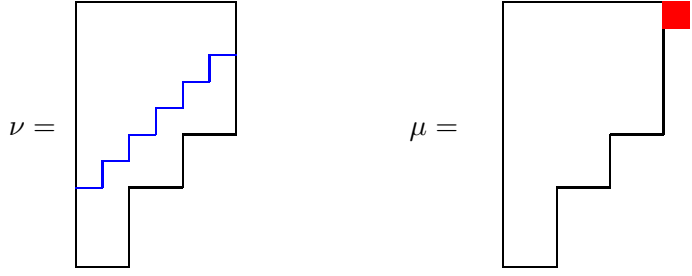
\begin{figure}[ht]
\centering
\begin{picture}(300,100)
\put(20,50){$\nu = $}
\put(45,0){\line(0,0){100}}
\put(45,100){\line(1,0){60}}
\put(105,50){\line(0,0){50}}
\put(95,80){\blue{\line(1,0){10}}}
\put(95,70){\blue{\line(0,0){10}}}
\put(85,70){\blue{\line(1,0){10}}}
\put(85,60){\blue{\line(0,0){10}}}
\put(75,60){\blue{\line(1,0){10}}}
\put(75,50){\blue{\line(0,0){10}}}
\put(65,50){\blue{\line(1,0){10}}}
\put(65,40){\blue{\line(0,0){10}}}
\put(55,40){\blue{\line(1,0){10}}}
\put(55,30){\blue{\line(0,0){10}}}
\put(45,30){\blue{\line(1,0){10}}}
\put(85,50){\line(1,0){20}}
\put(85,30){\line(0,0){20}}
\put(65,30){\line(1,0){20}}
\put(65,0){\line(0,0){30}}
\put(45,0){\line(1,0){20}}
\put(170,50){$\mu=$}
\put(205,0){\line(0,0){100}}
\put(205,100){\line(1,0){60}}
\put(265,50){\line(0,0){50}}
\put(245,50){\line(1,0){20}}
\put(245,30){\line(0,0){20}}
\put(225,30){\line(1,0){20}}
\put(225,0){\line(0,0){30}}
\put(205,0){\line(1,0){20}}
\multiput(265,90)(.5,0){20}{\red{\line(0,0){10}}}
\end{picture}
\caption{ A partition $\nu$ of type 1  and the corresponding partition $\mu$}
\label{fig:type1}
\end{figure}

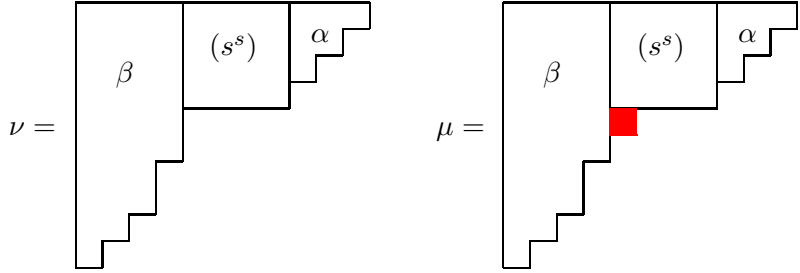
\begin{figure}[ht]
\centering
\begin{picture}(300,100)
\put(20,50){$\nu=$}
\put (45,0){\line(0,0){100}}
\put(45,100){\line(1,0){110}}
\put(85,40){\line(0,0){60}}
\put(125,60){\line(0,0){40}}
\put(85,60){\line(1,0){40}}
\put(155,90){\line(0,0){10}}
\put(145,90){\line(1,0){10}}
\put(145,80){\line(0,0){10}}
\put(135,80){\line(1,0){10}}
\put(135,70){\line(0,0){10}}
\put(125,70){\line(1,0){10}}
\put(75,40){\line(1,0){10}}
\put(75,20){\line(0,0){20}}
\put(65,20){\line(1,0){10}}
\put(65,10){\line(0,0){10}}
\put(55,10){\line(1,0){10}}
\put(55,0){\line(0,0){10}}
\put(45,0){\line(1,0){10}}
\put(133,85){$\alpha$}
\put(95,80){$(s^s)$}
\put(60,70){$\beta$}
\put(180,50){$\mu = $}

\put (205,0){\line(0,0){100}}
\put(205,100){\line(1,0){110}}
\put(245,40){\line(0,0){60}}
\put(285,60){\line(0,0){40}}
\put(245,60){\line(1,0){40}}
\put(315,90){\line(0,0){10}}
\put(305,90){\line(1,0){10}}
\put(305,80){\line(0,0){10}}
\put(295,80){\line(1,0){10}}
\put(295,70){\line(0,0){10}}
\put(285,70){\line(1,0){10}}
\put(235,40){\line(1,0){10}}
\put(235,20){\line(0,0){20}}
\put(225,20){\line(1,0){10}}
\put(225,10){\line(0,0){10}}
\put(215,10){\line(1,0){10}}
\put(215,0){\line(0,0){10}}
\put(205,0){\line(1,0){10}}
\put(293,85){$\alpha$}
\put(255,80){$(s^s)$}
\put(220,70){$\beta$}
\put(245,50){\red{\line(1,0){10}}}
\multiput(245,50)(.5,0){20}{\red{\line(0,0){10}}}
\end{picture}
\caption{ A partition $\nu$ of type 2 and the corresponding partition $\mu$ }
\label{fig:type2}
\end{figure}


The next two lemmas list some properties of partitions satisfying (C1). They follow directly from condition (C1).
\begin{lemma}  \label{properties of beta} Suppose $\nu=\b+(s^s)+\alpha$ satisfies the conditions in (C1).  Then, 
\begin{enumerate}
\item[(i)]  $\b_1=\b_2=\ldots =\b_{s+1}=\b_{s+2}$ and $\b_j\geq s+\b_1-j+2$ for $j \geq s+2$. 

\item[(ii)]  $\b'_j\geq s+\b_1-j+2$ for  $1\leq j \leq \b_1$.

\item[(iii)] If $\b \neq \emptyset$, then $\ell(\b)\geq s+\b_1+1$. 

\item[(iv)] If $\l=(\eta^c+\gamma, \sigma)$ is such that $\ds c^{\l}_{\nu'\nu^c}\neq 0$ and $\nu$ is not self-conjugate, then $\eta\neq \emptyset$.

\item[(v)] If $s=0$ (\textit{i.e.}, $\nu$ is of type 1), and $\l=(\eta^c+\gamma, \sigma)$ is such that $\ds c^{\l}_{\nu'\nu^c}\neq 0$, then  $\eta\neq \emptyset$.
\end{enumerate}

\end{lemma}

In the next lemma we consider the effect of removing rows or columns from a diagram on the type of the partition.

\begin{lemma} \label{remove} \ 

\begin{enumerate}\item[(i)]  If $\nu$ is a partition of type 1, then so is any partition $\nu^{(k)}$ obtained from $\nu$ by removing its first $k$ columns, $1\leq k \leq \nu_1-1$. 

\item[(ii)] If $\nu=\b+(s^s)+\a$ is a partition of type 2, then so is any partition $\nu^{(k)}$ obtained from $\nu$ by removing its first $k$ columns, $1\leq k \leq \b_1$. 

\item[(iii)]  If $\nu=\b+(s^s)+\a$ is a partition of type 2, then the partition $\nu_{(s)}$ obtained from $\nu$ by removing its first $s$ rows is of type 1.

\end{enumerate}

\end{lemma}

Conjecture \ref{conjecture} gives a criterion for  exactly when a partition covers another partition.  It might not be obvious  that  the conjecture implies that a partition $\nu$ satisfying condition (C1) is covered by a unique partition.  In what follows we prove this result. 

\begin{proposition}\label{lem:possible-types}\ 

\begin{enumerate}\item[(a)] If $\nu$ is a partition of type 2,  then  the  expression given in \text{\emph{(C1)}} is unique. 

\noindent \item[(b)]  If $\nu$ is a partition of type 1, then $\nu$ is not of type 2.  \textit{I.e.}, $\nu$ cannot be of type 1 and type 2 simultaneously.\end{enumerate} 

\end{proposition}
\begin{proof}
(a) Suppose that $\nu=\beta+(s_1^{s_1})+\alpha=\tilde{\b}+(s_2^{s_2})+\tilde{\a}$, with $s_1, s_2\geq1$,  $\alpha, \tilde{\a} \neq \emptyset$, and $\b, \tilde{\b}$ satisfying the conditions of (C1).  Moreover, assume that  the two decompositions are different. Then, either $s_1=s_2$ and  $\b_1 \neq \tilde{\b}_1$ (say $\b_1 < \tilde{\b}_1$), or $s_1\neq s_2$ (say $s_2<s_1$). In either case, by Lemma \ref{remove}, $\nu^{(\b_1)}$ is a partition of type 2.  One can easily see that this partition does not satisfy (iii) of Lemma \ref{properties of beta}.

(b)   Suppose that $\nu$ is simultaneously of type 1 and 2, \textit{i.e.},  $\nu=\beta+(s^{s})+\alpha=\tilde{\b}$ with $s\geq 1$, $\a \neq \emptyset$ and $\b, \tilde{\b}$ as in (C1). By Lemma \ref{remove}, $\nu^{(\b_1)}$ is a partition of type 1. However, this partition does not satisfy (iii) of Lemma \ref{properties of beta}. 
\end{proof}

If $\nu$ is of type 2, since the decomposition $\nu=\b+(s^s)+\a$ is unique, we refer to the Young diagram formed by the boxes in $(s^s)$ as the \textit{square in} $\nu$. 

Since the criterion in Conjecture \ref{conjecture} depends on the shape of $\nu$ or $\nu'$,  we  explore the relationship between partitions of type 1 and 2 and their conjugates. 

\begin{proposition} \label{lem:conjugate-types}\

\begin{enumerate}

\item[(a)] If $\nu$ is of type 1, then $\nu'$ is not of type 1.

\item[(b)] If $\nu=\b+(s^s)+\a$ is of type 2 with $\b_1\geq 1$, then $\nu'$ is not of type 1.

\item[(c)]  If $\nu=(s^s)+\a$, $\a \neq \emptyset$  (\textit{i.e.}, $\nu$ is of type 2 with $\b_1=0$), then $\nu'$ is of type 1.

\item[(d)] If $\nu$ and $\nu'$ are both of type 2, then $\nu$ is of the form $\nu=\beta+(s^{s})+\alpha$ with $s \geq 1$, $\a\neq \emptyset$ and $\b$ as in \text{\emph{(C1)}} satisfying $\b_{s+\b_1}=\b_1$.
\end{enumerate}
\end{proposition}

Before we prove Proposition \ref{lem:conjugate-types}, we give an example illustrating part $(d)$. 
\begin{example}Consider the  partition $\nu=(6,5,3,3,3,1)$. Then, both $\nu$ and  $\nu'=(6,5,5,2,2,1)$ are of type 2. We have $\b_1=3$, $s=2$, and $\b_{s+\b_1}=\b_5=3=\b_1$. The Young diagram of $\nu$ is given in Figure \ref{part(d)}.
\begin{figure}[h]
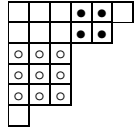

\centering
{\tiny \young(\hfill\hfill\hfill\bu\bu\hfill,\hfill\hfill\hfill\bu\bu,\cir\cir\cir,\cir\cir\cir,\cir\cir\cir,\hfil)}
\caption{ A partition of type 2 whose conjugate is of type 2}
\label{part(d)}
\end{figure}
\end{example}
The condition $\b_{s+\b_1}=\b_1$ ensures that the square in $\nu$ (marked with $\bu$ above) and the square in $\nu'$ (marked with $\cir$ above) meet but do not have any common edge. Then, assuming Conjecture \ref{conjecture} is true, the only cover of $\nu$ is the partition $\mu =(6,5,4,3,3,1)$. Moreover, $\mu'$ is the only cover of $\nu'$.

\begin{proof}
(a) Suppose $\nu$ is of type 1.  By Lemma \ref{properties of beta} (iii),  $\ell(\nu)\geq \nu_1+1$. Then, $\ell(\nu')=\nu_1\leq \ell(\nu)-1=\nu'_1-1$. Since $\nu'$ does not satisfy  Lemma \ref{properties of beta} (iii), it is not of type 1.   

(b)  Suppose $\nu=\beta+(s^{s})+\alpha$ is of type 2 with $\b_1\geq 1$ and $\nu'$ is of type 1. Then, by Lemma \ref{remove}, $\nu_{(s)}$ is of type 1 and $(\nu')^{(s)}$ is or type 1. Since $(\nu')^{(s)}=(\nu_{(s)})'$, this contradicts part (a). 

(c) If $\nu=(s^s)+\a$, $\a \neq \emptyset$, clearly $\nu'$ contains $(s,s,s-1,\ldots, 2,1)$.  

(d)   Suppose $\nu=\beta+(s^{s})+\alpha$  and $\nu'=\tilde{\b}+(t^t)+\tilde{\a}$ are  of type 2. Using Lemma \ref{remove}, we have $\nu^{(t)}$ is of type 2 and $(\nu^{(t)})_{(s)}$ is of type 1. Also, $(\nu')^{(s)}$ is of type 2 and $((\nu')^{(s)})_{(t)}$ is of type 1. Since  the conjugate of $(\nu^{(t)})_{(s)}$ is equal to $((\nu')^{(s)})_{(t)}$, by part (a), we must have $(\nu^{(t)})_{(s)}=\emptyset$. 
\end{proof}

Note that parts (b) and (c) of Proposition \ref{lem:conjugate-types} show that $\nu=\b+(s^s)+\a$ is of type 2 with conjugate  of type 1 if and only if $\b=\emptyset$. Equivalently, a partition $\nu$ is of type 1 with conjugate of type 2 if and only if    $(\nu_1^{\nu_1})\subseteq \nu$. 

\begin{theorem} \label{one-none} Assuming the validity of the main conjecture, a partition has either no cover or exactly one cover.
\end{theorem}

\begin{proof}

If $\nu=\b+(s^s)+\a$ is of type 2 with $\b \neq \emptyset$ and $\nu'$ is not of type 2, then by Proposition \ref{lem:possible-types}  and Proposition \ref{lem:conjugate-types} (b), there is only one place where a box can be added to $\nu$ to create a partition $\mu$ satisfying (C2). 

If $\nu$ is of type 1 and $(\nu_1^{\nu_1})\not \subseteq \nu$, then by Proposition \ref{lem:possible-types} (b) and Proposition \ref{lem:conjugate-types} (b) and (c), there is only one place where a box can be added to $\nu$ to create a partition $\mu$ satisfying (C2). 

Finally, if $\nu=\b+(s^s)+\a$ is of type 2 with $\b=\emptyset$  or $\nu$ is of type 2 with conjugate of type 2, then, by Proposition \ref{lem:conjugate-types} (c) and (d),   the box  added to $\nu$ to create $\mu$ as in (C2) is the same  as the box  added to $\nu'$ to create $\mu'$ as in (C2). 
\end{proof}

In the next theorem, we prove the interesting property that, if Conjecture \ref{conjecture} is true, then covering partitions $\mu$ are not covered by any partition.

\begin{theorem} \label{covers don't cover} Assuming the validity of the Conjecture \ref{conjecture}, if $\mu$ covers $\nu$, then $\mu$ is not covered by any partition of size $|\mu|+1$. 
\end{theorem}

\begin{proof} Since $\mu$ covers $\nu$, the partition $\nu$ must satisfy (C1) and the partition $\mu$ must satisfy (C2). We show that $\mu$ does not satisfy (C1). We consider different cases for all possibilities of $\nu$ and $\mu$. 

\noindent Case (I): $\nu$ is a partition of type 1. Then $\mu=\nu+(1)$. 
\begin{enumerate}
\item[(a)] Since $\mu_1=\mu_2+1$, the partition $\mu$ is not of type 1.  

\item[(b)] Since  $\ell(\mu')=\nu_1+1$ and $\mu'_1=\ell(\nu)\geq \nu_1+1$, we have $\ell(\mu')<\mu'_1+1$. By Lemma \ref{properties of beta} (iii), $\mu'$ is not of type 1. 

\item[(c)] If $\mu$ is of type 2, $\mu=\nu+(1)=\b+(s^s)+\a$.  By Lemma \ref{remove},  $\nu^{(\b_1)}$ is of type 1. However,  since $\nu^{(\b_1)}+(1)=(s^s)+\a$, the partition $\nu^{(\b_1)}$ does not satisfy Lemma \ref{properties of beta} (iii).  Thus, $\mu$ is not of type 2.  

\item[(d)] If $\mu'=\b+(s^s)+\a$ is of type 2, then $\nu^{(s)}$ and $\mu'_{(s)}$ are both of type 1. As shown above, in part (b), that this is impossible.  Thus, $\mu'$ is not of type 2.  

\end{enumerate} 

\noindent Case (II):  $\nu=\b+(s^s)+\a$ is a partition of type 2. Then, $\mu=\b+(s^s,1)+\a$.

\begin{enumerate}
\item[(a)]  If $\mu$ is of type 1, then so is $\mu^{(\b_1+1)}$, a partition of length $s$ and width at least $s$. This contradicts Lemma \ref{properties of beta} (iii). Thus $\mu$ is not of type 1. 

\item[(b)]   If $\mu'$ is of type 1, by Lemma \ref{remove},  $(\mu')^{(s)}$ and  $\nu_{(s)}$ are both of type 1.  As shown in  Case (I) (b) of this proof,  this is impossible. Thus $\mu'$ is not of type 1.

\item[(c)]  Suppose  $\mu$ is of type 2, \textit{i.e.}, $\mu =\tilde{\b}+(t^t)+\tilde{\a}$. We must have $s \geq t+2$, otherwise the decomposition of $\nu$ as a type 2 partition is not unique. Therefore $\b+(s^s) \subseteq \tilde{\b}$. By Lemma \ref{remove},  $\mu^{(\b_1+1)}$ is of type 2. This partition does not satisfy (iii) of Lemma \ref{properties of beta}. Thus, $\mu$ is not of type 2.    

\item[(d)]  Suppose $\mu'$ is of type 2, \textit{i.e.}, $\mu'=\tilde{\b}+(t^t)+\tilde{\a}$. By Lemma \ref{remove},  $(\nu_{(s)})^{(t)}$ and  $((\mu')^{(s)})_{(t)}$ are both of type 1. As shown in  Case (I) (b) of this proof,  this is impossible. Thus, $\mu'$ is not of type 2.  
\end{enumerate}
\end{proof}

Let $\mathcal{C}_1$ be the set of all partitions satisfying (C1), \textit{i.e.}, $\mathcal{C}_1$ consists of partitions $\nu$ such that $\nu$ or $\nu'$ is equal to $\b+(s^s)+\a$, with $\a, \b$ satisfying the following conditions: 

(i)  If $\b_1=i\geq0$ , then $\b$  contains $(i^{s+1}, i, i-1, i-2, \ldots, 1)$.

(ii)  $s\neq 0$ if and only if $\a \neq \emptyset$. 

Let $\mathcal{C}_2$ be the set of all partitions satisfying  (C2), \textit{i.e.}, $\mathcal{C}_2$ consists of partitions $\mu$ such that $\mu$ or $\mu'$  is equal to $\b+(s^s,1)+\a$, where $\b$ and $\a$ are as above. 

The next result follows immediately from Theorem \ref{covers don't cover}. 

\begin{corollary} The sets  $\mathcal{C}_1$ and  $\mathcal{C}_2$ are disjoint. 

\end{corollary}

We note that  $\mathcal{C}_1\cup \mathcal{C}_2$ is not equal to the set of all partitions. For example, partitions of the form $(s^s)$ are neither in $\mathcal{C}_1$ nor in $\mathcal{C}_2$. 


\section{Partitions of Type 1 }

In this section we prove that Conjecture \ref{conjecture} holds for some special cases of partitions of type 1.  In particular, we prove the conjecture for partitions $\nu$ of type 1 containing the rectangle $(\nu_1^{\nu_1-1})$ and those of width at most four. 

\subsection{Partitions of type 1 containing a large rectangle} 

Let $\nu=\b$ be a partition of type 1 with $\b_1=i$ and suppose $\b$ contains the rectangle $(i^{i-1})$. Therefore,  if $i\geq 2$, $\beta$ contains $(i^{i-1}, 2, 1)$, \textit{i.e.}, $\b'_i \geq i-1$. 
In the next theorem and its corollary we prove that $\nu$ is covered by $\mu=\nu+(1)$,  the partition obtained from $\nu$ by adding a box at the end of its first row, by showing that $\ds c_{\nu'\nu^c}^{\l}\leq c_{\mu'\mu^c}^{\l}$ for all partitions $\l$.  We use Corollary \ref{cor-sym} and consider Littlewood-Richardson tableau of shape $\gamma\times \sigma \times (\b/\eta)^*$ (respectively $\gamma\times \sigma \times ((\b+(1))/\eta)^*$) rather than of shape $ \gamma\times (\b/\eta)^*\times \sigma$ (respectively, $\gamma \times ((\b+(1))/\eta)^*\times\sigma$). This simplifies considerably the description  of tableaux and injections between sets of tableaux. 
For a partition $\l=(\eta^c+\gamma, \sigma)$ such that $\ds c_{\nu'\nu^c}^{\l}\neq 0$, we give an algorithm that assigns to each $T\in LR(\gamma\times\sigma\times (\beta/\eta)^*, \beta')$ a unique tableau in $LR(\gamma\times \sigma \times ((\b+(1))/\eta)^*,(\b+(1))')$. We denote by $x$  the box in the diagram  $\gamma\times \sigma \times ((\b+(1))/\eta)^*$ that is not in $\gamma\times \sigma \times (\b/\eta)^*$

\begin{figure}[ht]
\centering\hspace*{-2cm}
\begin{picture}(400,120)
\put(180,110){\small{$\gamma$}}
\put(170,90){\line(0,0){30}}
\put(170,120){\line(1,0){30}}
\put(200,110){\line(0,0){10}}
\put(190,110){\line(1,0){10}}
\put(190,100){\line(0,0){10}}
\put(180,100){\line(1,0){10}}
\put(180,90){\line(0,0){10}}
\put(170,90){\line(1,0){10}}
\put(130,73){\small{$\sigma$}}
\put(120,55){\line(0,0){30}}
\put(120,85){\line(1,0){30}}
\put(150,75){\line(0,0){10}}
\put(140,75){\line(1,0){10}}
\put(140,65){\line(0,0){10}}
\put(130,65){\line(1,0){10}}
\put(130,55){\line(0,0){10}}
\put(120,55){\line(1,0){10}}
\put(60,30){\small{$(\beta/\eta)^*$}}
\put(40,0){\line(1,0){30}}
\put(40,0){\line(0,0){30}}
\put(40,30){\line(1,0){10}}
\put(50,30){\line(0,0){10}}
\put(50,40){\line(1,0){10}}
\put(60,40){\line(0,0){10}}
\put(60,50){\line(1,0){40}}
\put(100,30){\line(0,0){20}}
\put(90,30){\line(1,0){10}}
\put(90,20){\line(0,0){10}}
\put(80,20){\line(1,0){10}}
\put(80,10){\line(0,0){10}}
\put(70,10){\line(1,0){10}}
\put(70,0){\line(0,0){10}}
\hspace*{.5cm}
\put(380,110){\small{$\gamma$}}
\put(370,90){\line(0,0){30}}
\put(370,120){\line(1,0){30}}
\put(400,110){\line(0,0){10}}
\put(390,110){\line(1,0){10}}
\put(390,100){\line(0,0){10}}
\put(380,100){\line(1,0){10}}
\put(380,90){\line(0,0){10}}
\put(370,90){\line(1,0){10}}
\put(330,73){\small{$\sigma$}}
\put(320,55){\line(0,0){30}}
\put(320,85){\line(1,0){30}}
\put(350,75){\line(0,0){10}}
\put(340,75){\line(1,0){10}}
\put(340,65){\line(0,0){10}}
\put(330,65){\line(1,0){10}}
\put(330,55){\line(0,0){10}}
\put(320,55){\line(1,0){10}}
\put(260,30){\small{$(\beta/\eta)^*$}}
\put(240,0){\line(1,0){30}}
\put(240,0){\line(0,0){30}}
\put(240,30){\line(1,0){10}}
\put(250,30){\line(0,0){10}}
\put(250,40){\line(1,0){10}}
\put(260,40){\line(0,0){10}}
\put(260,50){\line(1,0){40}}
\put(300,30){\line(0,0){20}}
\put(290,30){\line(1,0){10}}
\put(290,20){\line(0,0){10}}
\put(280,20){\line(1,0){10}}
\put(280,10){\line(0,0){10}}
\put(270,10){\line(1,0){10}}
\put(270,0){\line(0,0){10}}
\multiput(230,0)(.5,0){20}{\blue{\line(0,0){10}}}
\end{picture}
\caption{$\gamma\times \sigma\times (\beta/\eta)^*$ and $\gamma\times \sigma\times (\beta+(1)/\eta)^*$}
\label{fig:shapes}
\end{figure}
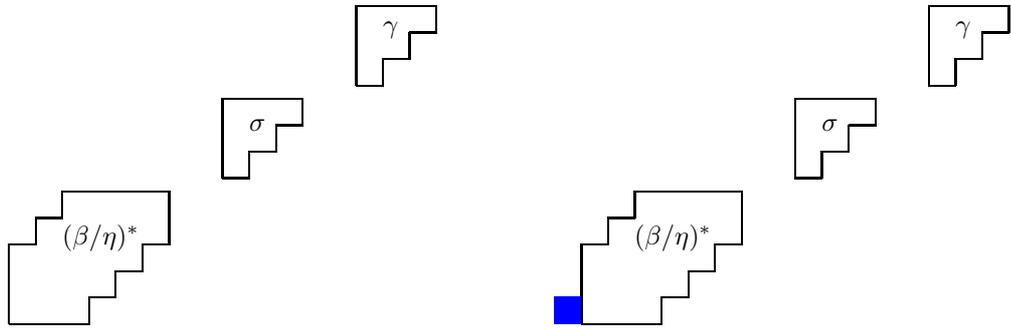

Figure \ref{fig:shapes} illustrates the shape of the diagrams to be  filled with the labels of the superstandard tableau of $\b'$ (for the diagram on the left)  respectively of $(\b', 1)$ (for the diagram on the right).  The box $x$ is  shaded blue.  The goal is to insert a label $i+1$ into a tableau of the shape shown on the left  so the result is a Littlewood Richardson  tableau of the shape on the right.   When we count columns in a skew shape such as those in Figure \ref{fig:shapes} we do so starting with the leftmost non-empty column. We only count non-empty columns.

\subsubsection*{\bf Algorithm:} Input a tableau $T \in LR(\gamma\times\sigma\times (\beta/\eta)^*, \beta')$.
\begin{itemize}
\item[{\bf (1)}] (Initializing step) Let  $j=\b_i$. If $\eta_1<i$ and $j<i$, set $q$ equal to the label in the first column and in row $i-j$ from the bottom in $T$.  

\item[{\bf (2)}] If $\eta_1=i$, let $v=i+1$. Go to (8).

\item[{\bf (3)}] If $\eta_1<i$, label $i+1$ replaces (bumps) the rightmost  label in the last row of $T$.
Go to (4).

\item[{\bf (4)}] If $j =i$  or $\eta_1\geq 2$, go to (7). Otherwise, go to (5).

\item[{\bf (5)}] If the leftmost label in the last row of $T$ is $i$, go to (6). Otherwise, go to (7). 

\item[{\bf (6)}] Let $v$ be the last bumped label. If ($q=j$ and $v=j+1$) or ($q=j+1$ and $v= j$), go to (7). Otherwise, $v$ bumps the label directly above it and go to (6). 

\item[{\bf (7)}] Let $v$ be the last bumped label. If there is no box directly to the left of  $v$, go to (8). Otherwise, $v$ bumps the label directly to its left and go to (7).

\item[{\bf (8)}] Place $v$ in box $x$. Output  tableau $T'$. {\bf STOP.}

\end{itemize}

\begin{example}
Consider the partition $\b=(6,6,6,6,6,3,1,1)$. Then $i=6, j=3$, and $\b'=(8,6,6,5,5,5)$. Let $\l=(\eta^c+\gamma, \sigma)$ with $\eta=(1,1,1,1,1), \gamma=(3,1)$, and $\sigma=(1)$. We show the effect of the algorithm on two tableaux  $T \in LR(\gamma\times\sigma\times (\beta/\eta)^*, \beta')$ in which  all labels $i$ are in the last row of $T$. In the first example, shown in Figure \ref{alg:1},  the label $q$ (marked in blue) in the first column and row $i-j=3$ from the bottom in $T$ is equal to $j=3$. In the second example, shown in Figure \ref{alg:2},  $q=j+1=4$. In $T$, the path of the insertion is marked with a  red line. In $T'$ the affected labels are marked in red. 

\begin{figure}[ht]
\centering
\begin{picture}(400,100)(0,0)
\put(25,20){\tiny \young(:::::::::111,:::::::::2,:::::::3,:::::1,:::::2,:::114,11222,23333,\bth4444,55555,66666)}
\put(150,50){$\longrightarrow$}
\put(210,20){\tiny \young(::::::::::111,::::::::::2,::::::::3,::::::1,::::::2,::::114,:11222,:23333,:\rf\rf\rf\rf\rfi,:5555\rs,\rth6666\rse)}
\put(50,5) {\small$T$} 
\put(245,5){\small $ T'$} 
\put(60,20){\red{\line(0,0){20}}}
\put(30,40){\red{\line(1,0){30}}}
\put(15,37.2){\red{$\longleftarrow$}}
\end{picture}
\caption{}
\label{alg:1}
\end{figure}

\begin{figure}[ht]
\centering
\begin{picture}(400,100)(0,0)
\put(25,20){\tiny \young(:::::::::111,:::::::::2,:::::::3,:::::1,:::::2,:::113,11222,33333,\bfo4444,55555,66666)} 
\put(150,50){$\longrightarrow$}
\put(210,20){\tiny \young(::::::::::111,::::::::::2,::::::::3,::::::1,::::::2,::::113,:11222,:\rth\rth\rth\rth\rf,:4444\rfi,:5555\rs,\rth6666\rse)}
\put(50,5){\small $T$} 
\put(245,5){\small $ T'$} 
\put(60,20){\red{\line(0,0){27}}}
\put(26,46){\red{\line(1,0){35}}}
\put(16,43){\red{$\longleftarrow$}}
\end{picture}
\caption{ }
\label{alg:2}
\end{figure}
\end{example}

\begin{theorem}\label{algorithm} Suppose  $\nu=\b$ is a partition of type 1 with $\b_1=i$ and such that $(i^{i-1})\subseteq \b$. Then, for each partition  $\l=(\eta^c+\gamma, \sigma)$ such that $\ds c_{\nu'\nu^c}^{\l}\neq 0$,  the above algorithm provides an injection 
$$LR(\gamma\times\sigma\times (\beta/\eta)^*, \beta') \hookrightarrow LR(\gamma\times \sigma \times ((\b+(1))/\eta)^*,(\b+(1))').$$

\end{theorem}

\begin{proof} Let $\l=(\eta^c+\gamma, \sigma)$ be a partition such that $\ds c_{\nu'\nu^c}^{\l}\neq 0$. We input a tableau  $T \in LR(\gamma\times\sigma\times (\beta/\eta)^*, \beta')$ into the algorithm. The algorithm produces a tableau $T'$. By construction, the shape of  $T'$ is $\gamma\times \sigma \times ((\b+(1))/\eta)^*$ and the type is $(\b+(1))'$. We  show that $T'$ is a SSYT whose reverse reading word is a lattice permutation and the map $T \rightarrow T'$   is an injection from $LR(\gamma\times\sigma\times (\beta/\eta)^*, \beta')$ into $LR(\gamma\times \sigma \times ((\b+(1))/\eta)^*,(\b+(1))')$.

If $\eta_1 =i$, then $$\gamma\times\sigma\times ((\beta+(1))/\eta)^*=\gamma\times\sigma\times (\beta/\eta)^*\times(1).$$ The algorithm simply places  the label $i+1$ in $x$. Clearly $T' \in LR(\gamma\times \sigma \times (\b/\eta)^*\times(1),(\b+(1))')$ and $T \rightarrow T'$ is an injection from $LR(\gamma\times\sigma\times (\beta/\eta)^*, \beta')$ into $LR(\gamma\times \sigma \times ((\b+(1))/\eta)^*,(\b+(1))')$.  In particular, this settles the case  $i=1$. 
 
Now suppose that $\eta_1<i$. If $i=2$ the algorithm replaces the label $2$ in the last row of $T$ by $3$ and places a label $2$ in the box $x$. The resulting tableaux is a SSYT. Since there are at least two labels $2$ in the superstandard tableau of $\b'$,  the reverse reading word of  $T'$ is a lattice permutation.  Suppose $T_1, T_2$ are two tableaux in  $LR(\gamma\times\sigma\times (\beta/\eta)^*, \beta')$ such that $T_1 \neq T_2$. Since $\eta_1\neq 2$, we must have $\b'_2=2$ and the last row of each tableaux $T_1, T_2$ consists of one box filled with $2$. Then, the tableaux $T_1, T_2$ differ in a row higher than the last row. Since in this case the algorithm only modifies the last row of a tableau in   $LR(\gamma\times\sigma\times (\beta/\eta)^*, \beta')$, it maps $T_1, T_2$ to different tableaux. Therefore, $T \rightarrow T'$ is an injection from $LR(\gamma\times\sigma\times (\beta/\eta)^*, \beta')$ into $LR(\gamma\times \sigma \times ((\b+(1))/\eta)^*,(\b+(1))')$. Together with the algorithm for the case $\eta_1=2$, this settles the case $i=2$. 
For the remainder of the proof, we assume that $i \geq 3$.

If $\b_i =i$ or $2\leq \eta_1<i$,  then the algorithm performs the loop in (7) without performing the loop in (6). Thus, the algorithm shifts all labels in the last row of $T$ one position to the left. The leftmost  label in this row is placed into the box $x$ and label $i+1$ is inserted into the rightmost box of the last row of $T$. This certainly creates a SSYT.  In this case, it is impossible for $T$ to contain all labels $i$ in the last row. Therefore, in the reverse reading word of $T'$, there is an $i$ before $i+1$ and thus the word  is a lattice permutation. It is straight forward to see that, in this case, for two different tableaux $T_1, T_2 \in LR(\gamma\times\sigma\times (\beta/\eta)^*, \beta')$   the algorithm produces different tableaux in $LR(\gamma\times \sigma \times ((\b+(1))/\eta)^*,(\b+(1))')$. Hence, we obtain an injection from $LR(\gamma\times\sigma\times (\beta/\eta)^*, \beta')$ into $LR(\gamma\times \sigma \times ((\b+(1))/\eta)^*,(\b+(1))')$.

For the remainder of the proof  assume that $\b_i<i$   and $\eta_1=1$. Then, we must have $\b_{i+1}=1$.  Let $T$ be a tableau in $LR(\gamma\times\sigma\times (\beta/\eta)^*, \beta')$ with at least one label $i$ in a row higher than the last.  Then the algorithm performs the loop in (7) without performing the loop in (6). As in the preceding paragraph, the resulting tableau is a SSYT whose reverse reading word is a lattice permutation. Moreover, if the algorithm is applied to two different tableaux $LR(\gamma\times\sigma\times (\beta/\eta)^*, \beta')$ such that  each tableaux has at least one label $i$ in a row higher than the last, the resulting tableaux are distinct.   
If  $T$ contains all labels $i$ in the last row, the algorithm performs the loop in (6) prior to the loop in (7). 
In this case, $T$ has to satisfy the following properties. (Recall that in this case $2 \leq j=\b_i\leq i-1$.)
\begin{itemize}
\item[(a)] For each $t$ such that $\min\{i,j+2\} \leq t \leq i$, row $i-t+1$ from the bottom has each box  filled with the label $t$ (\textit{i.e.}, the last row is filled with $i$, and, if $j <i-2$, the second to last row is filled with $i-1$, and so on until we reach row $(i-j-1)$  from the bottom). Moreover, $\ell(\eta)\geq \max\{1, i-j-1\}$. To see this, notice that in the superstandard tableau of $\b'$,  each label $i, i-1, \ldots, j+1$ appears an equal number of times while the number of labels $j$ is one more. 

\item[(b)] Because $\b'_1>\b'_2$, for each  label $p$ in the $(i-1)$-st column, the box directly to its right is either empty (\textit{i.e.}, not part of $(\b/\eta)^*$) or it contains a label greater than $p$. 

\item[(c)] The first column of $T$ contains the labels $1,2, \ldots, i$, with one label missing, in increasing order, from top to bottom. The missing label is between $1$ and $j+1$ but it cannot equal $j$ unless $j=2$. This is because, if $j>2$,  the number of labels $j-1$ equals the number of labels $j$ in the superstandard tableau of $\b'$. If label $j$ were missing from the first column of $T$, in the reverse reading word of $T$ we would have a  $j-1$ after the last $j$, which is impossible. 

\item[(d)]  Row $i-j$ from the bottom in $T$ is filled, from left  to right, with the labels $j, j+1, j+1, \ldots, j+1$ or with $j+1, j+1, \ldots, j+1$. Thus $q=j$ or $j+1$. The label directly above $q$  is exactly  $q-1$ (unless $j=2$, in which case, if $q=3$ the label directly above it could be $1$ or $2$).  If $q=j+1$ and $j>2$, row $i-j+1$ from the bottom in $T$ is filled with only $j$s.

\end{itemize}

If $T$ has all labels $i$ in the last row, the algorithm shifts labels up in column $i-1$ (the column of the rightmost $i$) until it reaches the row $i-j$ from the bottom, if $q=j$, or row $i-j+1$ form the bottom, if $q=j+1$. Then it continues  shifting  labels to the left in that row.  The leftmost label in this row is placed into the box $x$ and $i+1$ is inserted into the rightmost box of the last row. By properties (a), (b), and (d) above, $T'$ is a SSTY. If $j\neq 2$, the label placed in $x$ will always be $j$ and, since we have one more label $j$ than the number of labels $j+1$, the lattice permutation condition is not violated. If $j=2$, the label placed in $x$ is either $2$ or $1$. In either case the lattice permutation condition is not violated. Moreover, in the resulting tableau $T'$, there is a label $i$ directly above the label $i+1$. Thus $T' \in LR(\gamma\times \sigma \times ((\b+(1))/\eta)^*,(\b+(1))').$

Next, we will show that, if we apply the algorithm to   two different tableaux in $T_1, T_2 \in LR(\gamma\times\sigma\times (\beta/\eta)^*, \beta')$ with all labels $i$ in the last row, we obtain two distinct tableaux $T'_1, T'_2$. Consider first the case in which $j=\b_i=i-1$. Then the second to last row  of a tableau in $LR(\gamma\times\sigma\times (\beta/\eta)^*, \beta')$ with all labels $i$ in the last row must be filled with labels $i-1$. Therefore, $T_1$ and $T_2$ must differ in a row higher than the last two. Since the algorithm only affects labels in the last two rows, the resulting tableaux must be distinct. Now assume that $j=\b_i<i-1$. By property (a) above, $T_1$ and $T_2$ cannot differ in the last $i-j-1$ rows. By property (d) above, if $T_1$ and $T_2$ differ in row $i-j$ from the bottom, then they can only differ in the label $q$ (the leftmost label in this row). Suppose  $T_1$ has  $q=j$ while $T_2$ has $q=j+1$. Then, the $i$-th column of $T'_1$ is filled, from top to bottom, with the labels $1, 2, \ldots, j, j+2, j+3, \ldots i+1$.  The $i$-th column of $T'_2$ is filled, from top to bottom, with the labels $1, 2, \ldots, j-1, j+1, j+2, \ldots i+1$. Therefore, $T'_1\neq T'_2$. Now suppose  $T_1, T_2$ differ only in rows higher than row $i-j$ form the bottom. If they both have $q=j$,   the algorithm only affects the last $i-j$ rows of such a tableau and therefore $T'_1\neq T'_2$. If both $T_1, T_2$ have $q=j+1$, the algorithm only affects the last $i-j+1$ rows of such a tableau. Since tableau $T_1$ and $T_2$ do not differ in column $i-1$, and the algorithm shifts up in column $i-1$ and left in row $i-j+1$ from the bottom, we must have $T'_1\neq T'_2$.


Finally suppose that $T_1, T_2$ are tableaux in   $LR(\gamma\times\sigma\times (\beta/\eta)^*, \beta')$  such that $T_1$ has all labels $i$ in the last row and $T_2$ does not. Then the rightmost label in the second to last row is equal to $i$ in $T'_1$ and is equal to $i-1$ in $T'_2$. 

Thus,  if $\b'_i=i-1$ and $\eta_1=1$, the algorithm produces an injection from $LR(\gamma\times\sigma\times (\beta/\eta)^*, \beta')$ into $LR(\gamma\times \sigma \times ((\b+(1))/\eta)^*,(\b+(1))')$.  
\end{proof}

\begin{corollary}\label{sp-type1} If $\nu$ is a partition of type $1$ such that $(\nu_1^{\nu_1-1})\subseteq \nu$, then $\nu$ is covered by $\nu+(1)$.\end{corollary}


In what follows we  show that if $\nu$ is of type 1 containing the rectangle $(\nu_1^{\nu_1-1})$  and $\mu$ does not satisfy  (C2), then $s_{\mu'}s_{\mu^c}- s_{\nu'}s_{\nu^c}$ is not Schur-positive.

\begin{proposition} \label{nsp-type1-first-col} If $\nu$ is a partition of type $1$ such that $(\nu_1^{\nu_1-1})\subseteq \nu$, then $\nu$ is not covered by $(\nu,1)$, the partition obtained from $\nu$ by adding a box at the end of its first column.
\end{proposition}

\begin{proof} Let $\nu=\beta$ be as in the statement of the proposition. If $\b_1=i$ and $i\geq 2$, we have $(i^{i-1},2,1) \subseteq \b$. 

We claim that the smallest partition $\eta$ in lexicographic order such that 
 $\ds c_{\b' \b^c}^{\eta^c+\gamma, \sigma}\neq 0$ for some $\gamma, \sigma$ is  $$\eta=(\b_{i+1}, \b_{i+2}, \ldots, \b_{\ell(\b)}, \xi),$$ where $\xi=\emptyset$ if $i=1,2,3$ and, if $i \geq 4$, then $$\xi=\left\{ \begin{array}{ll} \emptyset & \mbox{if\ \ } i-1 \leq \b_i \leq i \\  & \\ (1^{i-1-\b_i}) & \mbox{if\ \ }  2 \leq \b_i<i-1.\end{array} \right.$$ Note that $\eta$ is obtained by reordering the parts of $\b/(\b \cap \b')$ to form a partition.
 
 Since $\ds c_{\b' \b^c}^{\eta^c+\gamma, \sigma}=|LR(\gamma\times\sigma\times (\beta/\eta)^*, \beta')|$, the columns of  $(\b/\eta)^*$ cannot be longer that $i$. Thus $(\b_{i+1}, \b_{i+2}, \ldots, \b_{\ell(\b)})\subseteq \eta$. Moreover, if $T \in LR(\gamma\times\sigma\times (\beta/\eta)^*, \beta')$, the reverse reading word of $T$ is a lattice permutation.  As explained in Remark \ref{horizontal-strips}, the rows of $T$ form horizontal strips in the superstandard tableau of $\b'$. 

 The largest horizontal strip in the superstandard tableau of $\b'$, read from right to left, ends in 
 \[\ldots, \b_i, \underbrace{i, i, \ldots, i}_\textrm{$i-1$ \mbox{times}}.\] 
 However, if $i \geq 4$ and $\b_i<i-1$, we can't have a label $\b_i$ (or smaller) at the end of the first column of $(\b/\eta)^*$ because this column has length $i-1$. Therefore, the last row of $(\b/\eta)^*$ contains only labels $i$. Then, since we are looking for the smallest $\eta$ in lexicographic order, the last column on $(\b/\eta)^*$ must have length less than $i$.  So, $\eta$ must contain $(\b_{i+1}, \b_{i+2}, \ldots, \b_{\ell(\b)},1)$. Similarly, since a label $\b_i$ cannot be placed in the  first column of $(\b/\eta)^*$ in any of the last $i-1-\b_i$ rows, $\eta$ must contain $(\b_{i+1}, \b_{i+2}, \ldots, \b_{\ell(\b)}, \xi)$ with $\xi$ as above.

Let $\eta= (\b_{i+1}, \b_{i+2}, \ldots, \b_{\ell(\b)}, \xi)$. Then, the partition $s_{(\eta^c+\eta')}$ appears in $s_{\b'}s_{\b^c}$ with multiplicity $1$, \textit{i.e.}, $\ds c_{\b' \b^c}^{\eta^c+\eta'}=1$. To see this, note that the only LR tableau of shape $\eta'\times (\b/\eta)^*$ and type $\b'$ is the one in which, $\eta'$ is filled to form a superstandard tableau and the columns of $(\b/\eta)^*$ are filled as follows from top to bottom: 

If $\b_i=i$, all columns are filled with $1, 2, \ldots, i$.

If  $\b_i=i-1$, the first column is filled with $1, 2, \ldots, i-1$ and  the next $i-1$ columns are filled with $1, 2, \ldots, i$.  

If $i \geq 4$ and $\b_i<i-1$, the first column is filled with $1, 2, \ldots, \b_i, \b_{i}+2, \ldots, i$, the next $i-\b_i-1$ columns are filled with $2, 3, \ldots, i$, the next $\b_i-1$ columns are filled with $1, 2, \ldots, i$, and  the right-most column is filled with $1, 2, \ldots, \b_i+1$.  

Since $\eta$ is the smallest partition in lexicographic order such that  $\ds c_{\b' \b^c}^{\eta^c+\gamma, \sigma}\neq 0$ for some $\gamma, \sigma$, it follows that $$\ds c_{(\b,1)' (\b,1)^c}^{\eta^c+\eta'}= 0.$$  Therefore, $(\b,1)$ does not cover $\b$. 
\end{proof}

\begin{theorem}\label{nsp-type1-restricted} If $\nu$ is a partition of type $1$ such that $(\nu_1^{\nu_1-1})\subseteq \nu$, then $\nu$ is not covered by any partition $\mu$ with $ \nu \subseteq \mu $, $|\mu|=|\nu|+1$, and $\mu \neq \nu+(1)$. 
\end{theorem}

\begin{proof} Let $\nu=\beta$ be as in the statement of the proposition. If $\b_1=i$ and $i \geq 2$, we have $(i^{i-1},2,1) \subseteq \b$.   Let $\mu$ be a partition as in the statement of the theorem. Then, $\mu$ is obtained from $\nu$ by adding a box at the end of its $p$th column, for some $1 \leq p \leq i$. If $p=1$, we are in the case of Proposition \ref{nsp-type1-first-col} and the statement of the theorem is true. Suppose $p\geq 2$. Consider the partition $\tilde{\nu}=\nu^{(p-1)}$ obtained from $\nu$ by removing its first $p-1$ columns. Then, by Proposition \ref{nsp-type1-first-col}, $\tilde{\nu}$ is not covered by $(\tilde{\nu},1)$. Let $j=\tilde{\nu}_1$. Since $\ds \tilde{\nu}_j=j$, it follows from the proof of Proposition \ref{nsp-type1-first-col} that $\tilde{\eta}= (\tilde{\nu}_{j+1}, \tilde{\nu}_{j+2}\ldots, \tilde{\nu}_{\ell(\tilde{\nu})})$ is such that $\ds c_{\tilde{\nu}' \tilde{\nu}^c}^{\tilde{\eta}^c+\tilde{\eta}'}=1$ and $\ds c_{(\tilde{\nu},1)'(\tilde{\nu},1)^c}^{\tilde{\eta}^c+\tilde{\eta}'}=0$.\medskip

Let $\varepsilon$ be the partition obtained by adjoining  $\tilde{\eta}$ to the right of the first $p-1$ columns of $\nu$, \textit{i.e.}, $\varepsilon=(\nu'_1, \nu'_2, \ldots, \nu'_{p-1}, (\tilde{\eta})')'$. Then, $\tilde{\nu}/\tilde{\eta}=\nu/\varepsilon$ and $|LR(\tilde{\eta}'\times (\tilde{\nu}/\tilde{\eta})^*, \tilde{\nu}')|= |LR(\varepsilon'\times ({\nu}/\varepsilon)^*, {\nu}')|$. Similarly, $|LR(\tilde{\eta}'\times ((\tilde{\nu},1)/\tilde{\eta})^*, (\tilde{\nu},1)')|=|LR(\varepsilon'\times ({\mu}/\varepsilon)^*, {\mu}')|$. Therefore,  $\ds c_{{\nu}' {\nu}^c}^{\varepsilon^c+\varepsilon'}=1$ and $\ds c_{{\mu}' {\mu}^c}^{\varepsilon^c+\varepsilon'}=0$. Thus, $\mu$ does not cover $\nu$. 
\end{proof}


\subsection{Partitions of type 1 with small width}

In this section we show that, if $\b_1\leq 4$, then Conjecture \ref{conjecture} holds. 

\begin{proposition}
 If $\nu$ is a partition of type 1 with $\nu_1=1, 2,$ or $3$, then Conjecture \ref{conjecture} is  true. Moreover, if $\nu$ is a partition of type 1 with $\nu_1=4$ and $\nu'_4 \geq 3$, then Conjecture \ref{conjecture} is  true. 
\end{proposition}

\begin{proof}
In each of these cases $\nu$  contains $(\nu_1^{\nu_1-1})$ and the conjecture follows from Theorems \ref{sp-type1} and \ref{nsp-type1-restricted}. 
\end{proof}

We now consider partitions $\nu$ of type 1 with $\nu_1=4$ and $\nu'_4=2$. To show that $\nu$ is covered by $\nu+(1)$, for each partition $\l=(\eta^c+\gamma, \sigma)$ such that $\ds c_{\nu'\nu^c}^{\l}\neq 0$, we give an algorithm that assigns to each $T\in LR(\gamma\times\sigma\times (\beta/\eta)^*, \beta')$ a unique tableau in $LR(\gamma\times \sigma \times ((\b+(1))/\eta)^*,(\b+(1))')$. The labels in the superstandard tableau of $\nu'$ are $1, 2, 3, 4$ and there are exactly two labels $4$.

\subsubsection*{\bf Algorithm:} Input a tableau $T\in LR(\gamma\times\sigma\times (\beta/\eta)^*, \beta')$.
 \begin{enumerate}
\item[{\bf (1)}] If $\eta_1=4$, let $v=5$. Go to (8). 
\item[{\bf (2)}] If $\eta_1<4$, then $5$ bumps the rightmost  label in the last row of $T$.  Go to (3).
\item[{\bf (3)}] If all labels  $4$ are in the last row of $T$, go to (4). Otherwise, go to (6). 
\item[{\bf (4)}] $4$ bumps the rightmost label in the second to last row of T. Go to (5)
\item[{\bf (5)}] Let $v$ be the last bumped label. If there is a label $3$ in a row higher than the second to last, then go to (6). Otherwise, $v$ bumps the rightmost label in the third to last row of $T$. Go to (6).  
\item[{\bf (6)}] Let $v$ be the last bumped label and let $(j,k)$ be its position in $T$.   If $v$ is strictly less than the label in box  $(j+1, k-1)$ or this box is empty, then $v$ bumps the label in position $(j, k-1)$ (\textit{i.e.}, directly to its left). Go to (7). Otherwise, $v$ bumps the label in position $(j+1, k-1)$. Go to (7).
\item[{\bf (7)}] Let $v$ be the last bumped label. If the label directly to its left is empty, go to (8). Otherwise, go to (6).
\item[{\bf (8)}] Place $v$ in box $x$. Output tableau $T'$. {\bf STOP.}
 \end{enumerate}

\begin{theorem}\label{algorithm2} Suppose  $\nu=\b$ is a partition of type 1 with $\b_1=4$ and $\b'_4=2$. Then, for each partition  $\l=(\eta^c+\gamma, \sigma)$ such that $\ds c_{\nu'\nu^c}^{\l}\neq 0$,  the above algorithm provides an injection 
$$LR(\gamma\times\sigma\times (\beta/\eta)^*, \beta') \hookrightarrow LR(\gamma\times \sigma \times ((\b+(1))/\eta)^*,(\b+(1))').$$

\end{theorem}

\begin{proof} Let $\lambda=(\eta^c+\gamma, \sigma)$, with $\eta \subseteq \beta$ and $|\gamma|+|\sigma|=|\eta|$, be a partition such that $c^{\lambda}_{\beta'\beta^c}\neq0$.  We input a tableau  $T\in LR(\gamma\times\sigma\times (\beta/\eta)^*, \beta')$ into the algorithm. The algorithm produces a tableau $T'$. 
By construction, the shape of the tableau $T'$  is $\gamma\times \sigma \times ((\b+(1))/\eta)^*$ and the type is $(\b+(1))'$. We  show that $T'$ is a SSYT whose reverse reading word is a lattice permutation and the map obtained   is an injection from $LR(\gamma\times\sigma\times (\beta/\eta)^*, \beta')$ into $LR(\gamma\times \sigma \times ((\b+(1))/\eta)^*,(\b+(1))')$.  

If $\eta_1=4$, by the  argument of Theorem \ref{algorithm} when $\eta_1=i$, $T'$ is a SSYT whose reverse reading word is a lattice permutation and $T \rightarrow T'$ is an injection from $LR(\gamma\times\sigma\times (\beta/\eta)^*, \beta')$ into $LR(\gamma\times \sigma \times ((\b+(1))/\eta)^*,(\b+(1))')$.

 If $\eta_1<4$ and  $T$ is a tableau with less than two labels $4$ in the last row, the algorithm skips steps (4) and (5) and performs the loop in (6). Thus, the algorithm  shifts all labels in the last row of $T$ one position to the left. The leftmost label in this row is placed into the box $x$ and label $5$ is inserted into the rightmost box in the last row of $T$. As explained in the proof of Theorem \ref{algorithm}, the resulting tableau is in  $LR(\gamma\times \sigma \times ((\b+(1))/\eta)^*,(\b+(1))')$ and applying the algorithm to two different tableaux $T_1, T_2 \in LR(\gamma\times\sigma\times (\beta/\eta)^*, \beta')$ with less than two labels $4$ in the last row produces different tableaux in $LR(\gamma\times \sigma \times ((\b+(1))/\eta)^*,(\b+(1))')$.

 We now consider the case in which $T \in LR(\gamma\times\sigma\times (\beta/\eta)^*, \beta')$ is a tableau with both labels $4$ in the last row. We will explicitly perform the algorithm on all such tableaux to see that the resulting tableau $T'$ is in $LR(\gamma\times \sigma \times ((\b+(1))/\eta)^*,(\b+(1))')$. Moreover, this will  allow us to conclude that, if $\eta_1<4$, the algorithm  produces an injection from $LR(\gamma\times\sigma\times (\beta/\eta)^*, \beta')$ into $LR(\gamma\times \sigma \times ((\b+(1))/\eta)^*,(\b+(1))')$.  

 Note that in all diagrams below we only show the relevant labels. Since the algorithm does not affect $\gamma$ and $\sigma$, we only show the $(\b/\eta)^*$ part of $T$. It is possible for boxes without labels to not be part of the skew shape (the reader should imagine such boxes as not being part of the tableau) and columns of  boxes without labels could be higher than shown, as long as the resulting diagram is a skew shape. Labels that are not marked do not change when the algorithm is performed.

If both labels $4$ are in the last row of $T$, we have $\eta_1=1$ or $2$. Since $\b'_1\geq 5$, we also have $\eta_2 \geq 1$. We consider several cases depending on the shape of $\eta$ and $\b$.  \medskip
 
\noindent{Case \bf ( I
) $\eta_1=1$}. Then the last row of $T$ is  {\tiny \young(q44) }. 
 Since the third column of $T$ ends in $4$ and $\b'_2\geq 4$, we must have $\b'_2=4$.
Therefore, the number of labels $2$ is equal to four and we have at least five labels $1$. Moreover, $\b'_3=3$ or $4$. We consider these two cases separately. 
\begin{itemize}
\item[(a)]  Case $\b'_3=4$. Then the first three columns of $T$ are 
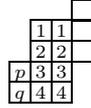
\begin{center} \tiny \young(:::\hfill,:11\hfill,:22\hfill,p33,q44)
 \end{center} 
 Since the number of labels $2$ equals the number of labels $3$, $q \neq 2$. Thus $q=3$ and $p=1$ or $2$. The result of applying the algorithm is shown below. 
 
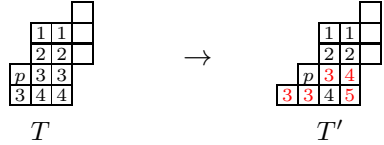
\begin{picture}(400,55)
\put(100,10) {\tiny \young(:::\hfill,:11\hfill,:22\hfill,p33,344)}
\put(165,25){$\rightarrow$}
\put(200,10) {\tiny \young(::::\hfill,::11\hfill,::22\hfill,:p\rth\rf,\rth\rth4\rfi)}
\put(108,-3){\small $T$}
\put(215,-3){\small $T'$} 
 \end{picture}
\vskip .1in
\noindent If $\eta_1=1$, $\b'_3=4$, and $T_1, T_2$ are  distinct tableaux in $LR(\gamma\times\sigma\times (\beta/\eta)^*, \beta')$ with both labels $4$ in the last row, then $T_1$ and $T_2$ can only differ in label $p$ or in a row higher than the second row from the bottom. Since the algorithm only affects the last two rows but does not change  the label $p$, the resulting tableaux $T'_1$ and $T'_2$ are different. 

\item[(b)]  Case $\b'_3=3$. Then the first three columns of $T$ are \begin{center} \tiny \young(:::\hfill,::1\hfill,:a2\hfill,pb3,q44)
 \end{center}
 If $q=2$, then $p=1$. We must also have $b=3$ because the number of  labels $2$ is one more than the number of  labels $3$.   
 Then, 
 
 
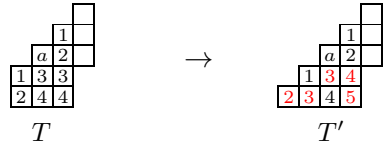
\begin{picture}(400,58)
\put(100,10){\tiny \young(:::\hfill,::1\hfill,:a2\hfill,133,244)}
\put(165,25){$\rightarrow$}
\put(200,10){\tiny \young(::::\hfill,:::1\hfill,::a2\hfill,:1\rth\rf,\rt\rth4\rfi)}
\put(108,-3){\small $T$}
\put(215,-3){\small $T'$}
\end{picture}
\vskip .1in
\noindent  Suppose now that  $q=3$.  Below we give the result of the algorithm depending on whether  $b=2$ or $3$. Note that if $b=3$,  there are no labels $3$ in a row higher than the second to last row in $T$. 
 
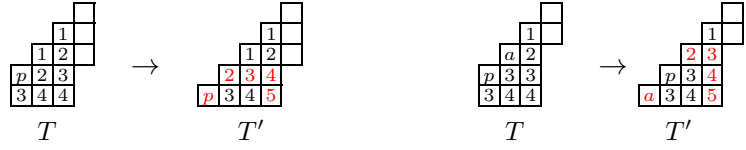
\begin{picture}(400,58)
\put(25,13){\tiny  \young(:::\hfill,::1\hfill,:12\hfill,p23,344)}
\put(70,25){$\rightarrow$}
\put(95,13){\tiny \young(::::\hfill,:::1\hfill,::12\hfill,:\rt\rth\rf,\rp34\rfi)}
\put(200,13){\tiny \young(:::\hfill,::1\hfill,:a2,p33,344)}
\put(245,25){$\rightarrow$} 
\put(260,13){\tiny \young(::::\hfill,:::1\hfill,::\rt\rth,:p3\rf,\ra34\rfi)}
\put(35,0){\small $T$} 
\put(110,0){\small $T'$} 
\put(210,0){\small $T$}
\put(270,0){\small $T'$}
\end{picture}

\noindent Suppose $\eta_1=1$, $\b'_3=3$, and $T_1, T_2$ are   distinct tableaux in $LR(\gamma\times\sigma\times (\beta/\eta)^*, \beta')$ with both labels $4$ in the last row. If $T_1, T_2$ differ in label $q$, then $T'_1, T'_2$ either differ in the second label (from the top) in column $4$ or else in the first label in column $2$ (recall that the first column of  each $T'_1, T'_2$ consists of only the box $x$).  If $T_1$ and $T_2$ have the same label $q$, and $q=2$, then they differ only in rows higher than the second to last row. Since the algorithm in this case only affects the last two rows, the resulting tableaux are different. If $T_1$ and $T_2$ have the same label $q$, and $q=3$, then $T'_1$ and $T'_2$ either differ in the  label in the second box in column four or else in at least one of the following labels:  label in box $x$ or  the first box of the second column or  a label in a row higher than the third to last row. 
\end{itemize}
 
\noindent {\bf Case ( II ) $\eta_1=2$}. Then, the last row of $T$ consists entirely of $4$. Because the type $\b'$ allows only for the labels $1,2,3,4$, we must have $\b'_3=3$ or $4$.  We distinguish between the two cases below. 
\begin{itemize}
\item[(a)] Case $\b'_3=4$. Then the first two columns of $T$ are 
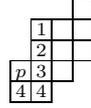
\[  \tiny \young(:::\hfill,:1\hfill\hfill,:2\hfill\hfill,p3\hfill,44)\]
 Below is the result of the algorithm depending on whether $\eta_2=1$ or $2$.
 
 
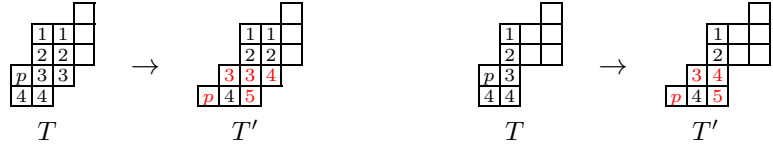
\begin{picture}(400,58)
 \put(25,13){\tiny \young(:::\hfill,:11\hfill,:22\hfill,p33,44)}
 \put(70,25){$\rightarrow$}
 \put(95,13){\tiny \young(::::\hfill,::11\hfill,::22\hfill,:\rth\rth\rf,\rp4\rfi)}
 \put(200,13) {\tiny  \young(:::\hfill,:1\hfill\hfill,:2\hfill\hfill,p3,44)}
 \put(245,25){$\rightarrow$}
 \put(270,13){\tiny \young(::::\hfill,::1\hfill\hfill,::2\hfill\hfill,:\rth\rf,\rp4\rfi)}
\put(35,0){\small $T$}
\put(105,0){ \small $T'$}
\put(210,0){\small $T$}
\put(280,0){\small$T'$}
\end{picture}

\noindent Suppose $\eta_1=2$, $\eta_2=1$ or $2$, $\b'_3=4$, and $T_1, T_2$ are  distinct tableaux in $LR(\gamma\times\sigma\times (\beta/\eta)^*, \beta')$ with both labels $4$ in the last row. Then  $T'_1, T'_2$ differ in the label in box $x$ or in a label in a row higher than the second to last row. 
  
\item[(b)] Case  $\b'_3=3$. Then  the first two columns of $T$ are 
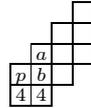
\[ \tiny \young(:::\hfill,::\hfill\hfill,:a\hfill\hfill,pb\hfill,44)\]
If  $\eta_2=2$, we have

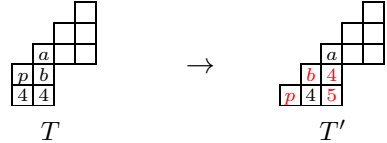
\begin{picture}(400,60)
\put(100,13){\tiny \young(:::\hfill,::\hfill\hfill,:a\hfill\hfill,pb,44)}
\put(165,25){$\rightarrow$}
\put(200,13){\tiny \young(::::\hfill,:::\hfill\hfill,::a\hfill\hfill,:\rb\rf,\rp4\rfi)}
\put(108,0){ \small $T$} 
\put(215,0){\small $T'$} 
\end{picture}
 
 \noindent Suppose $\eta_1=\eta_2=2$, $\b'_3=3$, and $T_1, T_2$ are  distinct tableaux in $LR(\gamma\times\sigma\times (\beta/\eta)^*, \beta')$ with both labels $4$ in the last row. Then $T'_1, T'_2$ will differ in at least one of the following labels: label in box $x$, or first label in the second column, or  a label in a row higher than the second to last row. 
 
 If $\eta_2=1$, then we must have $\b'_2=4$. If $p \neq 3$, there is a label $3$ in a row higher than the second to last row in $T$. Below we show the result of the algorithm depending on whether $p\neq 3$ or $p=3$. 
 
 
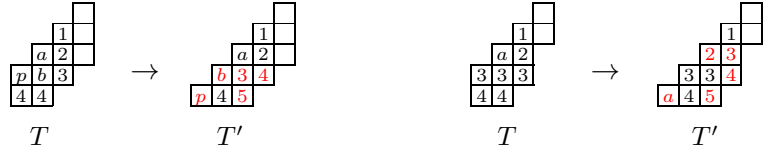
\begin{picture}(400,60)
\put(25,15){ \tiny \young(:::\hfill,::1\hfill,:a2\hfill,pb3,44)}
\put(70,25){ $\rightarrow$}
\put(95,15){\tiny  \young(::::\hfill,:::1\hfill,::a2\hfill,:\rb\rth\rf,\rp4\rfi)} 
\put(200,15){\tiny \young(:::\hfill,::1\hfill,:a2,333,44)}
\put(245,25){$\rightarrow$}
\put(270,15){\tiny \young(::::\hfill,:::1\hfill,::\rt\rth,:33\rf,\ra4\rfi)}
\put(35,0){\small $T$}
\put(105,0){\small $T'$}
\put(210,0){\small $T$}
\put(280,0){ \small $T'$}
 \end{picture}
 
\noindent Suppose  $\eta_1=2$, $\eta_2=1$, $\b'_3=3$, and $T_1, T_2$ are  two distinct tableaux in $LR(\gamma\times\sigma\times (\beta/\eta)^*, \beta')$ with both labels $4$ in the last row. Then $T'_1, T'_2$ will differ in the second box in  column four or else in one of the following labels: label in $x$, or the label in the first box of the second column, or the label in the first box of the third column, or a label in a row higher than the third to last row. 
\end{itemize}
The discussion above shows that if $T \in LR(\gamma\times\sigma\times (\beta/\eta)^*, \beta')$, the algorithm produces a tableau $T' \in LR(\gamma\times \sigma \times ((\b+(1))/\eta)^*,(\b+(1))')$. Moreover, we showed that if $T_1, T_2$ are two distinct tableaux in $LR(\gamma\times\sigma\times (\beta/\eta)^*, \beta')$ with both labels $4$ in the last row, then the resulting tableaux $T'_1, T'_2$ are distinct. We also showed that if $T_1, T_2$ are two distinct tableaux in $LR(\gamma\times\sigma\times (\beta/\eta)^*, \beta')$ with less than two  labels $4$ in the last row, then the resulting tableaux $T'_1, T'_2$ are distinct.

Now suppose $T_1, T_2$ are two distinct tableaux in $LR(\gamma\times\sigma\times (\beta/\eta)^*, \beta')$ such that $T_1$ has both labels $4$ in the last row and $T_2$ has at most one label $4$ in the last row. Then, if $\eta_1=1$ or $\eta_1=\eta_2=2$, $T'_1$ has a label $4$ directly above the label $5$ while  $T'_2$ has a label $2$ or $3$ directly above the label $5$. If $\eta_1=2$ and $\eta_2=1$, then in $T'_1$ the label in $x$ is less than or equal to the first label in the second column while in $T'_2$ the label in $x$ is strictly greater  than the first label in the second column. Therefore, in either case $T'_1 \neq T'_2$. 
\end{proof}

\begin{corollary} If $\nu$ is a partition of type 1 with $\nu_1=4$ and $\nu'_4=2$, then $\nu+(1)$ covers $\nu$. 
\end{corollary}

\begin{proposition}\label{failure-sp-i=4}  If $\nu$ is a partition of type 1 with $\nu_1=4$ and $\nu'_4=2$, then $\nu$ is not covered by any partition $\mu$ such that  $\nu \subset \mu$, $|\mu|=|\nu|+1$, and $\mu \neq \nu+(1)$. 
\end{proposition}

\begin{proof} To prove the proposition,  it is enough to show that $(\nu,1)$ does not cover $\nu$ by finding a partition $\eta$ such that $\ds c_{\nu'\nu^c}^{\eta^c+\eta'} \neq 0$ and  $\ds c_{(\nu,1)'(\nu,1)^c}^{\eta^c+\eta'} = 0$. If we can show this, the argument in the proof of Theorem \ref{nsp-type1-restricted} shows that $\nu$ is not covered by any   partition other than $\nu+(1)$.
 
  Let $\eta=(\b_5, \b_6, \ldots, \b_{\ell(\b)}, \xi)$, where $\xi=\emptyset$ if $\b'_3=3$, and $\xi=(1)$ if $\b'_3\geq 4$. (Note that $\eta$ is obtained by reordering the parts of $\b/(\b \cap \b')$ to form a partition.) One can easily check that $\ds c_{\b'\b^c}^{\eta^c+\eta'}=1$ and $\ds c_{(\b,1)'(\b,1)^c}^{\eta^c+\eta'}=0$.
\end{proof} 

The results above lead to the following theorem. 

\begin{theorem} \label{i=4} If $\nu$ is a partition of type 1 with $\nu_1=4$, then Conjecture \ref{conjecture} is true. 
\end{theorem}

The  proof of Theorem \ref{algorithm2} gives a glimpse into the difficulty of proving the conjecture in general by matching Littlewood-Richardson tableaux.\medskip

There are certain similarities between the algorithm above (when $\nu_1=4$ and $\nu'_4=2$) and the algorithm given at the beginning of this section for the case when $\nu$ is a partition of type 1 such that $(\nu_1^{\nu_1-1})\subseteq \nu$. In both cases, if all  labels $i$ (\textit{i.e.}, highest label) are in the last row of a tableau $T$, one creates $T'$ by bumping labels up and then to the left (and possibly South-West) according to certain rules.  Otherwise, one just bumps labels to the left in the last row. One might ask why such an algorithm would not work for partitions of higher width. Below we give an example where the natural generalization of the algorithm above fails. 

\begin{example} Consider the partition $\nu=(5,5,5,3,3,1)$ of type 1 with $\nu_1=5$. Let $\eta=(1^3)$, $\gamma=(2,1)$, and $\sigma=\emptyset$. In Figure \ref{i=5} we show a particular tableau $T \in LR(\gamma\times\sigma\times (\nu/\eta)^*, \nu')$. 

\begin{figure}[ht]
\centering
{\tiny \young(::::::11,::::::2,::::1,::112,::223,1333,2444,3555)} 
\caption{ A tableau $T$ for which the corresponding $T'$ is not a SSYT}
\label{i=5}
\end{figure} 
Notice that all labels $5$ are in the last row and all labels $4$ are in the last two rows, but there is a label $3$ in a row higher than the last three rows. If the algorithm in the proposition above would apply, $6$ would bump the rightmost $5$ in the last row, $5$ would bump the rightmost $4$ in the second to last row, $4$ would bump the rightmost $3$ in the third row from the bottom. Then, we would start bumping left in the third to last row. However, the leftmost $3$ is larger than the label $2$, which is directly to its SW. According to the algorithm, label $3$ would have to bump $2$ and label $2$ would be placed into $x$. The resulting tableau is not a SSYT.

\end{example}


\subsection{Failure of Schur-positivity}
The \textit{only if} part of Conjecture \ref{conjecture} for partitions of type 1 states that, if $\nu$ is of type 1 and $\mu$ is a partition such that $\nu \subseteq \mu$, $|\mu|=|\nu|+1$, $\mu \neq \nu+(1)$, then   $s_{\nu'\nu^c}-s_{\mu'\mu^c}$ is not Schur-positive. 

Based on the proofs of Theorem \ref{nsp-type1-first-col} and Proposition \ref{failure-sp-i=4}, and many explicit examples performed with Maple and Sage, we make the following conjecture.\medskip

\noindent {\bf Lex-minimality Conjecture:} \textit{If $\nu=\b$ is of type 1, the smallest partition $\eta$ in lexicographic order such that $\ds c^{\eta^c+\gamma, \sigma}_{\b'\b^c} \neq 0$, for some $\gamma, \sigma$, is given by reordering the rows of $\b/(\b\cap \b')$ to form a partition.  Moreover, for this $\eta$, we have $\ds c^{\eta^c+\eta'}_{\b'\b^c}=1$.} \medskip

The Lex-minimality conjecture implies that, a partition $\nu$  of type 1 is not covered by $(\nu,1)$. Using the same argument as in Theorem \ref{nsp-type1-restricted}, this would be enough to show that  a partition $\nu$ of type 1 could  only be covered by a partition $\mu$ satisfying (C2).

 In the remainder of this section we prove the failure of Schur-positivity  of the expression $s_{\nu'\nu^c}-s_{\mu'\mu^c}$
for   partitions $\nu$ of type 1 satisfying a symmetry condition and partitions $\mu$  such that $\nu \subseteq \mu$, $|\mu|=|\nu|+1$, $\mu \neq \nu+(1)$. 

Recall that an  outer corner of a Young diagram, and thus a partition, is a position (outside the diagram) such that, if we add a box in that position we still obtain a Young diagram.

\begin{defi} \label{c-s} A partition $\nu$ is called \textit{corner-symmetric} if for every outer corner $(k,j)$ of $\nu$ with $k>1$, there exists a non-empty partition $\eta_{k,j}$ such that $\nu/\eta_{k,j}$ is a self-conjugate (non-skew) partition and $\nu/\eta_{k,j}$ contains box $(k-1,j)$ of $\nu$ but does not contain any box from the first $j-1$ columns of $\nu$.  
\end{defi}
The partition in Figure \ref{corner-symmetric} (a) is corner-symmetric. It has three outer corners (marked with $\bullet$) in positions $(7,1)$, $(5,3)$, and $(3,5)$. We can take $\eta_{7,1}= (5,5)$, $\eta_{5,3}= (5,5,2,2,2,2,2)$, and $\eta_{3,5}=(5,4,4,4,2,2)$.
The partition in Figure \ref{corner-symmetric} (b) is not corner-symmetric. For the outer corner $(6,1)$ there is no partition $\eta_{6,1}$ satisfying Definition \ref{c-s}. 


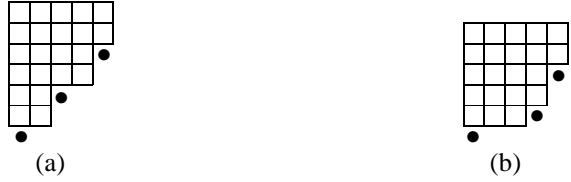
\begin{figure}[ht]
\begin{picture}(400,80)
\put(80,20){\tiny \yng(5,5,4,4,2,2)}
\put(113,44){$\bullet$}
\put(97,28){$\bullet$}
\put(82,13){$\bullet$}
\put(250,20){\tiny \yng(5,5,4,4,3)} 
\put(283,36){$\bullet$}
\put(275,21){$\bullet$}
\put(251,13){$\bullet$}
\put(90,3){\small (a)}
\put(260,3){\small (b)}
\end{picture}
\caption{  (a) Corner-symmetric and (b) non-corner-symmetric partitions }
\label{corner-symmetric}
\end{figure}

\begin{theorem} \label{cs-nsp}If $\nu$ is corner-symmetric and $\mu$ is such that $|\mu|=|\nu|+1$, $\nu \subseteq \mu$,  $\mu \neq\nu+(1)$, then $\mu$ does not cover $\nu$. 
\end{theorem}

\begin{proof} Let $\mu\neq \nu+(1)$ be a partition such that $|\mu|=|\nu|+1$ and $\nu \subseteq \mu$. Then $\mu$ is obtained from $\nu$ by adding a box at an outer corner $(k,j)$ with $k>1$. Let $\eta_{k,j}$ be the partition from Definition \ref{c-s} (which is not necessarily unique) and let  $\l=\eta^c_{k,j}+\eta'_{k,j}$. Then, $\l/\nu^c= \eta'_{k,j} \times (\nu/\eta_{k,j})^*$ and 
 the partition $\l$ appears in $s_{\nu'}s_{\nu^c}$ with multiplicity one but does not appear in $s_{\mu'}s_{\mu^c}$. Therefore $\mu$ does not cover $\nu$. 
\end{proof}

\begin{corollary} \label{cs-type1} A corner-symmetric partition of type $1$  is covered by at most one partition. \end{corollary} 


\section {Partitions of Type $2$}


Recall that $\nu$ is a partition of type $2$ if  $\nu=\beta+(s^s)+\alpha$, where $s\geq 1$, $\alpha\neq \emptyset$.  Moreover, by Proposition \ref{lem:possible-types} (a), this decomposition is unique.  For the rest of the section we set $i=\b_1$. Then,  $ (i^{s+1}, i, i-1, i-2, \ldots, 1)\subseteq \beta$. Consider the partition $\mu=\beta+(s^s,1)+\a$.  In this section, we prove that, for some particular partitions $\nu$ of type 2,  $\mu$ is the only partition covering $\nu$.

\begin{theorem} \label{square-alpha} Suppose $\nu$ is a partition of type $2$ with $\beta=\emptyset$, \textit{i.e.},  $\nu=(s^s)+\alpha$,  with $s\geq 1$,  $\alpha\neq\emptyset$. Then Conjecture \ref{conjecture} is true. 

\end{theorem}

\begin{proof} By Proposition \ref{lem:conjugate-types} (c), if $\nu$ is as in the statement of the theorem, then $\nu'$ is of type 1. Moreover, $\nu'_1=s$ and $(s^{s-1})\subseteq \nu'$. 
 The result now follows from Corollary \ref{sp-type1}, Theorem \ref{nsp-type1-restricted}, and Proposition \ref{conj}.  
\end{proof}

Before considering additional cases of partitions of type $2$, we introduce more notation and prove two helpful lemmas. 

Let $\lambda=(\eta^c+\gamma, \sigma)$ be a partition such that $c^{\lambda}_{\nu'\nu^c}\neq0$.  As we did with partitions of type 1 in the previous section, we attempt   to show that $c^{\lambda}_{\mu'\mu^c}\geq c^{\lambda}_{\nu'\nu^c}$ by matching each   tableau $T\in LR(\gamma\times\sigma\times (\nu/\eta)^*, \nu')$ with a unique tableau $T' \in LR(\gamma\times\sigma\times (\mu/\eta)^*, \mu')$. 

The shape  $\mu'$ is obtained from $\nu'$ by adding a box in position  $(i+1, s+1)$.  Thus, the type $\mu'$ provides us with  the same labels as the type $\nu'$ plus an additional label $i+1$. In $\gamma\times\sigma\times (\mu/\eta)^*$, we denote the box  $(\gamma\times\sigma\times (\mu/\eta)^*)/(\gamma\times\sigma\times (\nu/\eta)^*)$ by $x$. As before, we refer  to $x$ as the added box.  We refer to the row of $x$ in $T$ or $T'$ as $r_x$. Thus, $r_x$ is the $(s+1)$-st row from the bottom. 
In  $T$, we denote by $a$  the label directly to the right of $x$ (if it exists) and  by $b$ the label directly below $x$ (if it exists). Depending on the shape $\eta$,  there might not be a box directly below or directly to the right of $x$. 

\begin{example}In Figure \ref{type2-with-x}, we consider the diagram  in which $\nu=\b+(s^s)+\a$ with $\b=(2^5,1)$, $s=3$ and $\a=(2,1,1)$. Moreover,  $\eta=(3,1)$, $\gamma=\emptyset$ and $\sigma=(2,2)$. The figure shows $\gamma\times\sigma\times (\nu/\eta)^*$ with the place where the added box, $x$, would be placed. The label in the box to the right of $x$ is $a$ and the label in the  box below $x$ is $b$.  

\begin{figure}[ht]
\begin{picture}(400,70)
\put(150,0){\tiny \young(::::::::\hfil\hfil,::::::::\hfil\hfil,::::::\hfil,:::::\hfil\hfil,:::::a\hfil,:\hfil\hfil\hfil b\hfil\hfil,:\hfil\hfil\hfil\hfil\hfil,\hfil\hfil\hfil\hfil)} \put(183,25){\tiny$x$}
\end{picture}
\caption{}
\label{type2-with-x}
\end{figure}

If $\eta=(2^4)$, $\sigma=(4,4)$, and $\gamma=\emptyset$,  there is no box to the right of $x$ in $\gamma\times\sigma\times (\nu/\eta)^*$. If $\eta=(3,3,3,2)$, $\sigma=(4,4,3)$, and $\gamma=\emptyset$, there is no box directly below or directly to the right of $x$ in $\gamma\times\sigma\times (\nu/\eta)^*$. \end{example}


\begin{lemma}\label{obs1} Suppose $\nu=\b+(s^s)+\a$ is of type 2 with $\b_1=i$ and  $\lambda=(\eta^c+\gamma, \sigma)$ is such that  $c^{\lambda}_{\nu'\nu^c}\neq0$. Let $T\in LR(\gamma\times\sigma\times (\nu/\eta)^*, \nu')$. If $1\leq j\leq s$ and label $i+j$ appears in row $r$ of $T$, then $r$ is at most $j$ rows under $r_x$. In particular, the lowest row  in which label $i+1$ can appear is the $s$-th row from the bottom. 
\end{lemma}
\begin{proof} Since $T$ is of  type $\nu'$, each label  $i+j$,  $1 \leq j\leq s$, appears exactly $s$ times in $T$. Therefore, the lattice permutation condition forces $T$ to contain label $i+j+1$, $1 \leq j \leq s-1$, in a row below the row of the last label $i+j$. The statement of the lemma follows from the fact that there are at most $s$ rows under $r_x$. 
\end{proof} 
Note that it is possible for all labels $i+j$ (for some $1\leq j\leq s$) to be above $r_x$. 

\begin{corollary} Suppose $\nu=\b+(s^s)+\a$ is of type 2 with $\b_1=i$. For $\lambda=(\eta^c+\gamma, \sigma)$ such that  $c^{\lambda}_{\nu'\nu^c}\neq0$, let $T\in LR(\gamma\times\sigma\times (\nu/\eta)^*, \nu')$.  In $T$ there are at least $s$ letters $i$ occurring in  row $r_x$ or a row above it. 
\end{corollary} 

\begin{lemma}\label{obs2} Suppose $\nu=\b+(s^s)+\a$ is of type 2 with $\b_1=i$. For $\lambda=(\eta^c+\gamma, \sigma)$ such that $c^{\lambda}_{\nu'\nu^c}\neq0$, let $T \in LR(\gamma\times\sigma\times (\nu/\eta)^*, \nu')$. In $T$ there is at least one label $i+1$ in row $r_x$ or a row above it.
\end{lemma} 

\begin{proof} If all $i+1$ labels appear in the first row after $r_x$, then all labels $i+s$ appear in the last row of $T$. Since $\a \neq \emptyset$, there is a  label $i+s+1$ and it would have to be placed below the last row. 
\end{proof} 

 \begin{theorem}\label{one-col} Suppose $\nu=\beta+(s^s)+\a$ is a partition of type $2$ and $\b_1=i=1$. Then, $\mu=\beta+(s^s,1)+\a$ covers $\nu$. 
\end{theorem}
\begin{proof} Suppose $\lambda=(\eta^c+\gamma, \sigma)$  is such that   $\eta \subseteq \nu$ and $|\gamma|+|\sigma|=|\eta|$ and $c^{\lambda}_{\nu'\nu^c}\neq0$. Let $T\in LR(\gamma\times\sigma\times (\nu/\eta)^*, \nu')$. We will assign to $T$ a tableau $T'\in LR(\gamma\times\sigma\times (\mu/\eta)^*, \mu')$ in a unique way. The tableau $T$ contains $s$ labels $2$ whereas the tableau $T'$ contains $s+1$ labels $2$.

Since the partition $(1^{s+2})$ is contained in $\b$, if there is a box in $T$ directly to the right of $x$, then $a \geq 2$. If there is no label $1$ in a row below $r_x$, then either there is no box directly below $x$ or  $b>2$. (If $b=2$ and all labels $1$ appear in row $r_x$ or a row above it, then all $s$ boxes in the row of $b$ and to the left of $b$ must be filled with label $2$.  This contradicts Lemma \ref{obs2}.)  

The tableau $T'$ is constructed from $T$ in the following way.  If there is no label $1$ in a row below $r_x$, place a label $2$ into the box $x$. Otherwise, find the highest, leftmost  label $1$ below $r_x$  and replace it by $2$. Then, place a label $1$ into  the box $x$.  By Lemma \ref{obs1}, there is no label $2$ in a row lower than the first row below $r_x$. Thus,  the tableau $T'$ obtained above is a SSYT. By  construction, its reverse reading word is a lattice permutation.

It is straightforward to see that this construction gives an injection \[LR(\gamma\times\sigma\times (\nu/\eta)^*, \nu') \hookrightarrow LR(\gamma\times \sigma \times (\mu/\eta)^*,\mu').\]
  \end{proof}
Recall that $\nu_{(s)}$ denotes the partition obtained from $\nu$ by removing its first $s$ rows.
\begin{proposition} \label{nsp-type2} Let $\nu=\beta+(s^s)+\a$ be a partition  of type $2$ with $\b_1=i$ and such that $\nu_{(s)}$ is corner-symmetric. Then $\nu$ is not covered by any partition $\mu$ with $\nu \subseteq \mu$, $|\mu|=|\nu|+1$, and $\mu\neq\beta+(s^s,1)+\a$. 
\end{proposition}
\begin{proof}
If $\mu$ is obtained from $\nu$ by adding a box at an outer corner $(k,j)$ at the end of the $j$th column with $1 \leq j \leq i$, let $\eta_{k,j}$ be the partition obtained by applying Theorem \ref{cs-nsp} to $\nu_{(s)}$, 
and let $\eta=(\nu_1, \nu_2, \ldots, \nu_s, \eta_{k,j})$. Then, $\ds c_{\nu'\nu^c}^{\eta^c+\eta'}=1$ and $\ds c_{\mu'\mu^c}^{\eta^c+\eta'}=0$. Therefore, $\mu$ does not cover $\nu$.

Suppose now that $\mu$ is obtained from $\nu$ by adding a box at an outer corner at the end of the $j$th column with $i+s+1\leq j\leq \nu_1$, \textit{i.e.}, the added box is at the end of one of the columns of $\a$. We denote by $\tilde{\a}$ the partition obtained from $\a$ by adding this box. 
We will show that $\nu'$ is not covered by $\mu'$. Then, Proposition \ref{conj} will imply that $\nu$ is not covered by $\mu$.
Consider  $\nu'=(\b',s^s,\a')$ and let $\chi=(\nu')_{(i)}$. Then, $\chi=(s^s,\a')$ is a partition of type 1 with $\chi'_{\chi_1}\geq \chi_1$. By Proposition \ref{nsp-type1-restricted}, $\chi$ is not covered by $\tilde{\chi}=(s^s,\tilde{\a}')$. Moreover, the proof of Proposition \ref{nsp-type1-restricted} gives  a partition $\eta$ for which $\ds c_{\chi'\chi^c}^{\eta^c+\eta'}=1$ and $\ds c_{\tilde{\chi}'\tilde{\chi}^c}^{\eta^c+\eta'}=0$. Let $\tilde{\eta}=(\b', \eta)$. It is easy to see that $\ds c_{\nu(\nu')^c}^{\tilde{\eta}^c+\tilde{\eta}'}=1$ and $\ds c_{\mu(\mu')^c}^{\tilde{\eta}^c+\tilde{\eta}'}=0$. Therefore $\mu'$ does not cover $\nu'$. 
\end{proof} 

\begin{corollary} Suppose $\nu=\beta+(s^s)+\a$ is a partition of type $2$ and $\b_1=i=1$. Then $\nu$ is not covered by any partition $\mu$ with $\nu \subseteq \mu$, $|\mu|=|\nu|+1$, and $\mu\neq\beta+(s^s,1)+\a$. 
\end{corollary}

We remark that  the proof of Proposition \ref{nsp-type2} shows that if we can prove that a  partition  of type 1 is not covered by any partition  not satisfying (C2), then  it follows that a partition of type 2 is not covered by any partition not satisfying (C2). Therefore, establishing the Lex-minimality conjecture would prove the \textit{only if} part of Conjecture \ref{conjecture} for partitions of type 1 and 2. 


\section{Non-Type $1$ or $2$ Partitions}

In this section we consider some partitions $\nu$ such that neither $\nu$ nor $\nu'$ is of type $1$ or $2$, \textit{i.e.}, partitions not satisfying (C1). The next  proposition sheds light on the necessity of $\a\neq \emptyset $ in the definition of type $2$ partitions.

\medskip

\begin{proposition} \label{squareplusbeta}Let $\nu=\beta+(s^s)$ with $s\geq 1$, $\beta_1=i\geq 0$ and such that $\b$ contains  $(i^{s+1}, i, i-1, i-2, \ldots, 1)$. Then, $\nu$ is not covered by $\mu=\b+(s^s, 1)$. Moreover, $\nu$ is not covered by $\mu=\b+(s^s)+(1)$ either.\end{proposition}

\begin{proof} If $\mu=\b+(s^s, 1)$ or $\mu=\b+(s^s)+(1)$, then $\lambda=\beta^c+\beta'$ appears with multiplicity $1$ in $s_{\nu'}s_{\nu^c}$ but does not appear in  $s_{\mu'}s_{\mu^c}$.  (Note that if $\b=\emptyset$, then $\l=(m^m)$ is the square in which we take the complement.) \end{proof}  

\begin{corollary}\label{square-part} Let $\nu=(s^s)$, $s\geq 1$. Then $\nu$ is not covered by any partition $\mu$.
\end{corollary} 

\begin{proof} This follows from Proposition \ref{squareplusbeta} with $\b=\emptyset$ and the fact that the only partitions $\mu$ such that $\nu \subseteq \mu$, $|\mu|=|\nu|+1$ are $\mu=(s^s, 1)$ and $\mu=(s^s)+(1)$.  
\end{proof}

We now consider  partitions $\nu$ such that both $\nu$ and  $\nu'$ are corner symmetric. We show that they are neither of type 1 nor of type 2 and are not covered by any partition $\mu$ with $\nu\subseteq \mu$   and $|\mu|=|\nu|+1$. 

\begin{proposition} \label{nu-nu'-cs}  If $\nu$ is a partition such that both $\nu$ and $\nu'$ are corner-symmetric, then $\nu$ is neither of type 1 nor of type 2. 
\end{proposition}

\begin{proof} First we show that  a partition $\nu$  of type 2 is not corner-symmetric. Let $\nu$ be a partition of type 2, $\nu=\b+(s^s)+\a$, with $\b_1=i$, and consider the outer corner $(s+1,i+1)$. Suppose $\nu$ is corner-symmetric. Then, by Definition \ref{c-s}, there exists a partition $\eta_{s+1,i+1}$ such that $\nu/ \eta_{s+1,i+1}$ is a self-conjugate (non-skew) partition containing the box $(s, i+1)$. Because of the shape of $\nu$, the partition $\nu/ \eta_{s+1,i+1}$ has last part equal to $s$ and, in order  to be self-conjugate,  it must end in $s$ parts of length $s$. Since $\a \neq \emptyset$, this is impossible.  Thus $\nu$ is not corner-symmetric. 

Next, we show that, if $\nu$ is a partition such that  $\nu'$ is corner-symmetric, then $\nu$ is not of type 1. Suppose that $\nu$ is of type 1 and $\nu'$ is corner-symmetric. Let $j=\b'_i$. Then $\nu'$ has an  outer corner at $(i+1, 1)$. For any partition $\eta_{i+1,1}$ as in Definition \ref{c-s}, the last row of $\nu'/ \eta_{i+1,1}$ has length $j$. This implies that $\nu'/ \eta_{i+1,1}$, and therefore $\nu'$, must end in $j$ rows of length $j$, \textit{i.e.}, $\nu'_{i-j+1}=j$. However, since $\nu$ is of type 1, by Lemma \ref{properties of beta} (ii), we have $\nu'_{i-j+1}\geq i-(i-j+1)+2=j+1$. Thus $\nu$ cannot be of type 1 if $\nu'$ is corner-symmetric.  
\end{proof} 
The next proposition follows directly from Proposition \ref{nu-nu'-cs} and Theorem \ref{cs-nsp}. 

\begin{proposition} If $\nu$ is a partition such that both $\nu$ and $\nu'$ are corner-symmetric, then $\nu$ is not covered by any partition $\mu$ such that $\mu\subseteq \nu$ and $|\mu|=|\nu|+1$.  
\end{proposition} 

\begin{corollary} If $\nu$ is self-conjugate, then $\nu$ is not covered by any partition $\mu$ such that $\mu\subseteq \nu$ and $|\mu|=|\nu|+1$. 
\end{corollary}

Notice that a partition $\nu$ such that both $\nu$ and $\nu'$ are corner-symmetric is not necessarily self-conjugate. For example $\nu=(6,5,5,5,4,4,3,3,3)$ is corner-symmetric and so is $\nu'$. However, $\nu$ is not self-conjugate. 

\begin{corollary} The staircase partition $\delta_i=(i, i-1,   \ldots, 2,1)$, with $i\geq 1$,  is not covered by any partition.
\end{corollary}

\section{Final Remarks}

The proof of the main conjecture in complete generality  will likely involve some new combinatorial ideas. As the proofs of Theorem \ref{algorithm} and Theorem \ref{algorithm2} suggest, trying to match tableaux to show Schur-positivity seems to involve complicated insertion algorithms. The larger the width of the shape $\b$, the more exception rules need to be introduced. In an attempt to prove the main conjecture we have also tried (unsuccessfully) using the Jacobi-Trudi identity (see \cite{stanleyEC2}) and, separately, the Pl\"ucker relations (see \cite{K}). \medskip

We note that if $\nu$ and $\mu$ are such that $\nu \subseteq \mu$ and $|\mu|=|\nu|+1$, then McNamara's necessary conditions for Schur-positivity \cite{PM} for $s_{\m'}s_{\m^c}-s_{\nu'}s_{\nu^c}$ are satisfied. The particular cases of Conjecture \ref{conjecture} proved in this article provide another set of examples showing that the conditions are not sufficient.


\medskip




\begin{thebibliography}{000}


\bibitem {BBR} F. Bergeron, R. Biagioli and M. Rosas, {\em Inequalities between Littlewood-Richardson coefficients}, J. Comb. Theory, Series A. Vol. 113, issue 4, (2006), 567- 590.

\bibitem{BM} F. Bergeron and P. McNamara,  {\em Some positive differences of products of Schur functions}, preprint 2004.  
math.CO/0412289. 

\bibitem{FFLP} S. Fomin, and W. Fulton, C.-K Li, and Y.-T. Poon, {\em Eigenvalues, singular values, and Littlewood-Richardson coefficients}, Amer. J. Math.  127 (2005), 101-127. 

\bibitem{KWW} R. C. King, T. A. Welsh, S. J. van Willigenburg, {\em Schur positivity of skew Schur function differences and applications to ribbons and Schubert classes},  J. Algebraic Combin. 28 (2008), no. 1, 139 - 167. 

\bibitem{Ki} A. N. Kirillov. \emph{An invitation to the generalized saturation conjecture}. Publ. Res. Inst. Math. Sci., 40(4):1147Ð1239, 2004.

\bibitem{K} M. Kleber, \emph{Pl\"ucker Relations on Schur Functions}, J. Algebraic Combin. 13 (2001), no. 2, 199Ð211.



\bibitem{LPP} T. Lam and A. Postnikov and P. Pylyavskyy, {\em Schur positivity and Schur log-concavity}, Amer. J. Math., 129 (2007), 1611-1622. 

\bibitem{L} D. E. Littlewood, {\em The Kronecker product of symmetric group representations}, J. London Math. Soc. 31 (1956), 89-93.

\bibitem{ma} I. G. Macdonal, "Symmetric Functions and Hall polynomials", sec. ed.
\emph{Oxford University Press}, 1995.

\bibitem{PM} P. W. McNamara, \emph{Necessary conditions for Schur-positivity}. J. Algebraic Combin. 28 (2008), no. 4, 495-507.

\bibitem{MvW1} P. W. McNamara and  S. van Willigenburg, \emph{Positivity results on ribbon Schur funtions diffrences}, European J. Combin. 30, (2009), no. 5, 1352-2369.

\bibitem{MvW2} P. McNamara and  S. van Willigenburg,  \emph{Maximal supports and Schur-positivity among connected skew shapes}. European J. Combin. 33 (2012), no. 6, 1190Ð1206.

\bibitem {O} A. Okounkov, \emph{Log-concavity of multiplicities with applications to characters of $U(\infty)$}, Adv. Math., 127 no. 2 (1997), 258-282. 

\bibitem {STW} T. Scharf, J-Y Thibon, and B.G. Wybourne. {\em Powers of the Vandermonde determinant and the quantum Hall effect}, J. Phys. A: Math. Gen. 27 (1994) 4211-4219.

\bibitem {stanleyEC2} R. Stanley. \emph{Enumerative Combinatorics} Vol 2,
\emph{Cambridge Univ. Press}, 1999.

\end{thebibliography}
\end{document}